\documentclass{amsart} 
\usepackage[english]{babel}
\usepackage[utf8]{inputenc}

\usepackage{amsmath}
\usepackage{amsthm}
\usepackage{amssymb}
\usepackage{tikz}
\usepackage{csquotes}
\usepackage{imakeidx} 
\usepackage{appendix}
\usepackage{tikz-cd}
\usepackage{verbatim}
\usepackage{enumitem}
\usepackage{multicol}
\usepackage{aliascnt}
\usepackage[backref=page]{hyperref}
\usepackage{cleveref}
\usepackage{mathtools}
\usepackage[normalem]{ulem}
\usepackage[a4paper,margin=2.5cm,footskip=35pt]{geometry}

\hypersetup{colorlinks=true,linkcolor=blue,anchorcolor=blue,citecolor=blue}

\newtheorem{thm}{Theorem}[section]
\crefname{thm}{Theorem}{Theorems} 

\newaliascnt{prop}{thm}
\newtheorem{prop}[prop]{Proposition}
\aliascntresetthe{prop}
\crefname{prop}{Proposition}{Propositions}

\newaliascnt{lemma}{thm}
\newtheorem{lemma}[lemma]{Lemma}
\aliascntresetthe{lemma}
\crefname{lemma}{Lemma}{Lemmas}

\newaliascnt{cor}{thm}
\newtheorem{cor}[cor]{Corollary}
\aliascntresetthe{cor}
\crefname{cor}{Corollary}{Corollaries}

\newaliascnt{conj}{thm}

\aliascntresetthe{conj}
\crefname{conj}{Conjecture}{Conjectures}

\newaliascnt{question}{thm}

\aliascntresetthe{question}
\crefname{question}{Question}{Questions}

\theoremstyle{definition}
\newaliascnt{defin}{thm}
\newtheorem{defin}[defin]{Definition}
\aliascntresetthe{defin}
\crefname{defin}{Definition}{Definitions}

\theoremstyle{definition}              
\newaliascnt{nota}{thm}                
      
\aliascntresetthe{nota}                
\crefname{nota}{Notation}{Notations}

\newaliascnt{rmk}{thm}
\newtheorem{rmk}[rmk]{Remark}
\aliascntresetthe{rmk}
\crefname{rmk}{Remark}{Remarks}

\numberwithin{equation}{section}

\newcommand{\Z}{\mathbb Z}

\renewcommand{\P}{\mathbb P}
\newcommand{\Spec}{\operatorname{Spec}}

\newcommand{\mc}[1]{\mathcal{#1}}
\newcommand{\cl}{\overline}

\renewcommand{\phi}{\varphi}

\newcommand{\on}[1]{\operatorname{#1}}

\newcommand{\RC}{\mathrm{RC}}
\newcommand{\RCE}{\mathrm{RCE}}

\title{Comb smoothing and local triviality of homogeneous spaces over a relative curve}

\address{Institut Camille Jordan, Université Claude Bernard Lyon 1, Bâtiment Jean Braconnier 21 Av. Claude Bernard, 69100 Villeurbanne and Institute of Mathematics \enquote{Simion Stoilow} of the Romanian Academy, 21
Calea Grivitei Street, 010702 Bucharest, Romania}

\author{Margot Bruneaux}
\email{bruneaux@math.univ-lyon1.fr}

\thanks{The first author was supported by the project \enquote{Group schemes, root systems, and related representations} funded by the European Union - NextGenerationEU through Romania’s National Recovery
and Resilience Plan (PNRR) call no. PNRR-III-C9-2023-I8, Project CF159/31.07.2023, and coordinated by the Ministry of Research, Innovation and Digitalization (MCID) of Romania, and by the LABEX MILYON (ANR-10-LABX-0070) of Université de Lyon, within the program \enquote{France 2030} (ANR-11-IDEX-0007) operated by the French National Research Agency (ANR)}

\author{Federico Scavia}

\makeatletter
\@namedef{subjclassname@2020}{\textup{2020} Mathematics Subject Classification}
\makeatother

\definecolor{MyDarkGreen}{rgb}{0.0,0.5,0.0}

\address{CNRS\\ Institut Galil\'ee \\ Universit\'e Sorbonne Paris Nord \\ 93430, Villetaneuse, France}
\email{scavia@math.univ-paris13.fr}


\begin{document}

\begin{abstract}
Let $R$ be a Henselian local ring, let $\kappa$ be the residue field of $R$, let $C$ be a smooth projective curve over $R$ with geometrically connected fibers, let $G$ be a reductive $C$-group with isotrivial radical torus $\mathrm{rad}(G)$, and let $E\to C$ be a $G$-torsor. We show that, if either the kernel of the central isogeny $G^{\mathrm{sc}}\times_C \mathrm{rad}(G)\to G$ is \'etale over $C$ or $\kappa$ is large, the Zariski-local triviality of $E_\kappa\to C_\kappa$ implies the Zariski-local triviality of $E\to C$. We also prove an averaged form of this result, assuming only that $\mathrm{rad}(G)$ is isotrivial, as well as a variant for projective homogeneous spaces under no restrictions on $G$. As consequences, we obtain a local-global principle for torsors over function fields of curves over Henselian discrete valuation rings, strengthening work of Gille--Parimala--Suresh and a Henselian version of a theorem of Drinfeld--Simpson. Our proofs are geometric and rely on compactifications of torsors and on a relative and arithmetic version of the comb smoothing technique, which we develop in detail, building on work of Koll\'ar and Graber--Harris--Starr.
\end{abstract}

\vspace*{-0.5cm}
\maketitle

\vspace*{-0.5cm}
\tableofcontents

\section{Introduction}

Let $C$ be a smooth projective relative curve over a scheme $S$, and let $G$ be a reductive group scheme over $C$. A natural question is whether $G$-torsors over $C$ become Zariski-locally trivial after passage to a cover $S' \to S$ in some Grothendieck topology. In the \'etale setting, this statement is related to uniformization results for the stack of $G$-torsors over a curve; in the Nisnevich setting, it provides a key geometric input for local-global principles for torsors over semi-global function fields.

\subsection{Local triviality and uniformization} The starting point is the work of Beauville--Laszlo \cite{beauville1994conformal}: given a section $\Sigma\subset C$ of $C\to S$, they proved \cite[Lemma 3.5]{beauville1994conformal} that every $\on{SL}_n$-torsor on $C$ is trivial on $C\setminus \Sigma$ Zariski-locally on $S$, and then used this to prove the uniformization of moduli spaces of vector bundles with trivial determinant over a curve. Drinfeld--Simpson
\cite{drinfeld1995b-structures} extended the uniformization statement to split semisimple groups $G$: given a section of $C\to S$, every $G$-torsor becomes trivial on the complement of the section after a suitable fppf base change $S'\to S$; if the order of $\pi_1(G)$ is invertible on $S$, then one may take $S'\to S$ to be \'etale. This was extended by Heinloth
\cite{heinloth2010uniformization}. For a parahoric group scheme
$G$ over a fixed curve $C_0$ over a field $k$, Heinloth proved that if $x\in C_0$ is a closed point,
then every family of $G$-torsors over $C_0\times_k S$ becomes trivial on
$(C_0\setminus\{x\})\times_k S'$, for some faithfully flat base change $S'\to S$. The theorems of Beauville--Laszlo, Drinfeld--Simpson and Heinloth underlie the one-point uniformization theorem for the stack of $G$-torsors over a curve \cite[Theorem 4.1.6]{zhu2017introduction}, and show how questions about the local triviality of
torsors fit in a broader framework involving loop groups, patching, and moduli
stacks of bundles.

For an arbitrary reductive $C$-group $G$, one cannot expect the $G$-torsor $E\to C$ to become trivial on the complement of a given section $\Sigma$ of $C\to S$, not even after a faithfully flat base change $S'\to S$, as can already be seen for the case of $\mathbb{G}_m$-torsors (equivalently, line bundles) on constant $1$-parameter families of elliptic curves. In the arbitrary reductive case, the most natural replacement of this condition that one can aim for is Zariski-local triviality. This is already considered in the work of Drinfeld--Simpson, who show that, for a split reductive group $G$, every $G$-torsor on $C$ admits a reduction to a Borel subgroup of $G$ after a suitable surjective \'etale base change on $S$, and hence in particular becomes Zariski-locally trivial after such a base change.

\subsection{Local triviality over a Henselian base} There is also a finer local triviality problem. Let $C\to S$ be a
smooth projective curve with geometrically connected fibers, let $G$ be a reductive $C$-group, let $s\in S$, and let $E$ be a $G$-torsor on $C$. If the fiber $E_s\to C_s$ is Zariski-locally trivial, does this local triviality spread to a neighborhood of $s$ in a suitable Grothendieck topology? For
this question, the natural topology on the base scheme $S$ is not the \'etale topology but the Nisnevich topology. Indeed, an \'etale neighborhood of $s$ may replace
the residue field $\kappa(s)$ by a finite separable extension, potentially adding
new sections, while a Nisnevich neighborhood preserves
the residue field at $s$. Moreover, considering the Nisnevich topology is relevant for local-global principles over semi-global fields. Indeed, if $R$ is a Henselian discrete valuation ring and $C\to\operatorname{Spec}R$ is a regular proper curve with function
field $F$, then local-global questions for $G_F$-torsors are governed by the
relation between the generic fiber, the special fiber, and the completions of
$F$. A Nisnevich-local triviality theorem for torsors on the model $C$ is
therefore a natural geometric input for proving local-global principles over semi-global function fields: this is the perspective taken by Gille--Parimala--Suresh \cite{Gille2021}.

Since the question is Nisnevich-local at the point $s\in S$, we may
pass to the Henselization of the local ring of $S$ at $s$. The problem then
takes the following form. Let $R$ be a Henselian local ring with
residue field $\kappa$, let $C\to\operatorname{Spec}R$ be a smooth projective
relative curve, and let $E$ be a $G$-torsor over $C$. If the closed fiber
$E_\kappa\to C_\kappa$ is Zariski-locally trivial, does it follow that $E$ is Zariski-locally trivial? 

Let $\on{rad}(G)$ be the radical torus of $G$, let $G^{\mathrm{sc}}$ be the simply connected cover of the derived subgroup of $G$, and let $\mu(G)$ be the kernel of the central isogeny $G^{\mathrm{sc}}\times_C \mathrm{rad}(G)\to G$, which is a finite flat $C$-group. Under the assumption that $\mu(G)$ is \'etale and that $\mathrm{rad}(G)$ becomes quasi-split after a finite \'etale morphism $C'\to C$ of degree invertible in $\kappa$, Gille--Parimala--Suresh \cite[Theorem 7.2]{Gille2021} prove that $E\to C$ is Zariski-locally trivial. As a consequence, they prove the local-global principle for $G$-torsors over a Henselian discrete valuation ring in many new cases; see \cite[Corollary 7.8]{Gille2021}. 

The assumption that $\mu(G)$ be \'etale is a mild restriction on the residue
characteristic. In the semisimple case it amounts to requiring that the order of
$\pi_1(G_\kappa)$ be invertible in $\kappa$, and is the Nisnevich analogue of the
corresponding hypothesis in Drinfeld--Simpson's \'etale-local uniformization
theorem. By contrast, the hypothesis on the radical torus is of a different
nature. Gille--Parimala--Suresh assume that the torus $\mathrm{rad}(G)$ becomes quasi-split after a finite \'etale cover $\widetilde C\to C$ of degree prime to
the residue characteristic. This condition enters through their treatment of the toral case \cite[Theorem 2.4]{Gille2021} via the proper base change theorem in \'etale cohomology.

\subsection{Main results} 

In this paper, we study the problem of Nisnevich-local triviality using geometric methods. Our first theorem is the following strengthening of the main result of \cite{Gille2021}.

\begin{thm}\label{thm:main}
  Let $R$ be a Henselian local ring with residue field $\kappa$, and let
  $C$ be a smooth projective $R$-curve with geometrically connected
  fibers, and let $G$ be a reductive $C$-group whose radical
  $\operatorname{rad}(G)$ is isotrivial and such that the natural central
  isogeny $G^{\mathrm{sc}}\times_C \operatorname{rad}(G)\to G$
  has \'etale kernel. If two $G$-torsors $E_1$ and $E_2$ over $C$ are
  Zariski-locally isomorphic over the special fiber $C_\kappa$, then they
  are Zariski-locally isomorphic over $C$.
\end{thm}

Recall that every torus over a geometrically unibranch locally Noetherian scheme is isotrivial; see \cite[Expos\'e X, Th\'eor\`eme 5.16]{SGA3IInew}. Therefore, under the assumptions of \Cref{thm:main}, if $R$ is Noetherian and geometrically unibranch, the radical torus $\on{rad}(G)$ is automatically isotrivial.

\Cref{thm:main} was proved by Gille--Parimala--Suresh \cite[Theorem 7.2]{Gille2021} when $\mathrm{rad}(G)$ becomes quasi-split over a finite \'etale extension $C'\to C$ of degree invertible in $\kappa$; note that this implies that $\mathrm{rad}(G)$ is isotrivial. In particular, \Cref{thm:main} is new already if $G$ is an isotrivial $C$-torus (in which case $\mu(G)$ is trivial). 

When $\kappa$ is large (that is, every $\kappa$-curve with a smooth $\kappa$-point has a Zariski-dense set of $\kappa$-points), our methods allow us to remove the assumption on $\mu(G)$.

\begin{thm}\label{thm:large-field}
  Let $R$ be a Henselian local ring with large residue field $\kappa$, and let
  $C$ be a smooth projective $R$-curve with geometrically connected
  fibers, and let $G$ be a reductive $C$-group whose radical
  $\operatorname{rad}(G)$ is isotrivial. If two $G$-torsors $E_1$ and $E_2$ over $C$ are
  Zariski-locally isomorphic over the special fiber $C_\kappa$, then they
  are Zariski-locally isomorphic over $C$.
\end{thm}

We also prove an averaged version of \Cref{thm:main}, assuming only that the radical torus of $G$ is isotrivial.

\begin{thm}\label{thm:main-average}
Let $R$ be a Henselian local ring with residue field $\kappa$, and let
  $C$ be a smooth projective $R$-curve with geometrically connected
  fibers, and let $G$ be a reductive $C$-group whose radical
  $\operatorname{rad}(G)$ is isotrivial. If two $G$-torsors $E_1$ and $E_2$ over $C$ are
  Zariski-locally isomorphic over the special fiber $C_\kappa$, then there exists a finite collection of finite \'etale Henselian local $R$-algebras $\{R_i\}_{i\in I}$ of collectively coprime degrees such that, for each $i \in I$, the base changes $(E_1)_{R_i}$ and $(E_2)_{R_i}$ are Zariski-locally isomorphic.
\end{thm}

Theorems \ref{thm:main} and \ref{thm:large-field} have the following consequence for the local-global principle over function fields of curves over Henselian discrete valuation rings.

\begin{cor}\label{cor:local-global}
Let $R$ be a Henselian discrete valuation ring with residue field
$\kappa$, let $C$ be a smooth projective $R$-curve with geometrically
connected fibers, and let $F$ be the function field of $C$. Let $G$ be a
reductive $C$-group.
Assume that at least one of the following conditions holds:
\begin{enumerate}
  \item[(i)] the isogeny $G^{\mathrm{sc}}\times_C \operatorname{rad}(G)\to G$ has \'etale kernel;
  \item[(ii)] the field $\kappa$ is large.
\end{enumerate}
Then the local-global principle holds for $G_F$-torsors over $F$ with
respect to the discrete valuations associated with the codimension-one
points of $C$.
\end{cor}

When $R=\Z_p$ is the ring of $p$-adic integers, for some prime $p$, \Cref{cor:local-global} is due to Colliot-Th\'el\`ene, Parimala and Suresh; see \cite[Theorem 4.8]{colliot2012patching}. It was proved by Harbater--Hartmann--Krashen under the assumption that $G_F$ is $F$-rational; see \cite[Theorem 3.7]{harbater2009applications}. When $\mu(G)$ is \'etale and the radical torus of $G$ becomes quasi-split over a finite \'etale extension $C'\to C$ of degree invertible in $\kappa$, it is due to \cite[Corollary 7.8]{Gille2021}.

\subsection{Projective homogeneous spaces} 
Our methods also apply to relative projective homogeneous spaces. More generally, we prove the following result (see \Cref{cor_prop_principal_Weil}).

\begin{thm}\label{intro:cor_prop_principal_Weil}
Let $R$ be a Henselian local ring with residue field $\kappa$, let
$C$ be a smooth projective $R$-curve with geometrically connected fibers, and let $Z \subset C_{\kappa}$ be a subscheme finite over $\kappa$. Let $X$ be a smooth projective $R$-scheme, let $p\colon X\to C$ be an $R$-morphism, and assume that $p_\kappa$ admits a rational section. Denote by $s\colon C_\kappa\to X_\kappa$ its unique extension to a section. Assume moreover that there exists a nonempty open subscheme $C_\kappa^\circ\subset C_\kappa$ such that the restriction $p_\kappa^{-1}(C_\kappa^\circ)\to C_\kappa^\circ$ is smooth and has separably rationally connected fibers.
\begin{enumerate}
    \item There exists a finite collection of finite \'etale Henselian local $R$-algebras
$\{R_i\}_{i\in I}$ of collectively coprime degrees such that, for each $i\in I$, the base change $X_{R_i} \to C_{R_i}$ has a section $s_i$ such that $s_i|_{Z_{\kappa_i}} = s_{\kappa_i}|_{Z_{\kappa_i}}$, where $\kappa_i$ is the residue field of $R_i$.
    \item Assume further that the field $\kappa$ is large. Then there exists a section $s'$ of $p \colon X \to C$ such that ${s'}|_{Z}=s|_{Z}$.
\end{enumerate}
\end{thm}

We originally proved \Cref{intro:cor_prop_principal_Weil} assuming that every fiber of the morphism $p\colon X\to C$ is smooth and separably rationally connected. This is enough for our purposes. The current stronger form of \Cref{intro:cor_prop_principal_Weil}, where only a general fiber of $p$ is supposed to be smooth and separably rationally connected, was suggested to us by János Koll\'ar. We include it here for possible future use; the increased generality requires only minor modifications to the proofs.

As a consequence of \Cref{intro:cor_prop_principal_Weil}, we prove a Henselian analogue of the aforementioned theorem of Drinfeld--Simpson \cite[Theorem~1]{drinfeld1995b-structures}.
 
 \begin{cor}\label{cor:drinfeld-simpson}
Let $R$ be a Henselian local ring with residue field $\kappa$, let $C$ be a smooth projective $R$-curve with geometrically connected fibers, let $G$ be a quasi-split reductive group over $C$, and let $B$ be a Borel subgroup of $G$. Let $E$ be a $G$-torsor over $C$, and suppose that $E_\kappa\to C_\kappa$ is Zariski-locally trivial. 
\begin{enumerate}
    \item There exists a finite collection of finite étale Henselian local $R$-algebras $\{R_i\}_{i\in I}$ of collectively coprime degrees such that $E_{R_i}$ admits a reduction of structure group to $B\times_C C_{R_i}$ for every $i\in I$.
    \item If $\kappa$ is large, then $E$ admits reduction of structure to $B$.
\end{enumerate}
\end{cor}

Applying \Cref{cor:drinfeld-simpson} to $G=\mathrm{PGL}_{n,C}$ for all $n\geqslant 1$, one deduces that the specialization map between Brauer--Azumaya groups $\mathrm{Br}_{\mathrm{Az}}(C)\to \mathrm{Br}_{\mathrm{Az}}(C_\kappa)$ is injective. This is due to Colliot-Thélène--Ojanguren--Parimala \cite[Theorem~1.8]{ColliotTheleneOjangurenParimala2002}: using Artin's approximation theorem, they proved that the map $\mathrm{Br}_{\mathrm{Az}}(C)\to \mathrm{Br}_{\mathrm{Az}}(C_\kappa)$ is bijective.

\subsection{Comb smoothing and proofs of the main results} 

The works described above studied uniformization and local triviality via techniques involving moduli stacks, loop groups, and patching. In contrast, our proofs of Theorems \ref{thm:main}, \ref{thm:large-field} and \ref{thm:main-average} are geometric; more precisely, they are based on a relative and arithmetic version of the comb smoothing technique, whose development takes up the majority of this paper.

For the proofs of Theorems \ref{thm:main} and \ref{thm:large-field}, by a twisting argument it suffices to show that, for every $G$-torsor $E\to C$ such that $E_\kappa\to C_\kappa$ is Zariski-locally trivial, the $G$-torsor $E\to C$ is Zariski-locally trivial. Since $\mathrm{rad}(G)$ is isotrivial, a theorem of Nath \cite{nath2026compactification} implies the existence of a fiberwise dense smooth proper compactification $\overline{E}$ of $E$ over $C$; see \Cref{sec:6} and Appendix~\ref{sec:appendix}. By the valuative criterion for properness, a rational section of $E_\kappa\to C_\kappa$ extends uniquely to a section $s\colon C_\kappa\to \overline{E}_\kappa$. Equivalently, $s$ defines a $\kappa$-point of the Weil restriction $\mathrm{Res}_{C/R}(\overline{E})$, which by a theorem of Grothendieck is represented by an $R$-scheme locally of finite type. If this $\kappa$-point were smooth, we could lift it by Hensel's lemma to an $R$-point of $\mathrm{Res}_{C/R}(\overline{E})$. The strategy is therefore to replace this $\kappa$-point by a smooth point whose corresponding section agrees with the initial one on any prescribed finite closed subset of $C_\kappa$. In order to accomplish this, we prove an arithmetic version of the comb smoothing technique of \cite{graber2003families}.

\begin{thm} \label{thm:add_enough_teeth}
Let $k$ be an infinite field, let $D$ be a geometrically connected smooth projective $k$-curve, let $Z \subset D$ be a finite subscheme, let $X$ be a smooth projective $k$-scheme, let $p\colon X\to D$ be a morphism, and assume that there exists a nonempty open subscheme $D^\circ\subset D$ such that $p^{-1}(D^\circ)\to D^\circ$ is smooth and has separably rationally connected fibers, and let $s \colon D \to X$ be a section of $p$. Then, for every integer $N \geqslant 0$, there exists an integer $m_N \geqslant N$, a comb $\mathcal{C}$ with handle $D$ and teeth $C_1, \dots, C_{m_N}$ attached to $D$ at points of $D^\circ$ distinct from those of $Z$, and a morphism $f \colon \mathcal{C} \to X$ smoothable fixing $f(Z)$, such that $f|_D = s$ and such that, for all $1 \leqslant i \leqslant m_N$, $p \circ f|_{C_i}$ is constant and the restriction $f|_{C_i} \colon C_i \to X$ is almost very free.
\end{thm}

To prove \Cref{thm:add_enough_teeth}, we first establish a smoothing theorem for combs embedded in $X$ satisfying a certain cohomological vanishing condition; see \Cref{thm:smoothing}. We then construct a comb satisfying this condition inside $X$ with handle $D$; see \Cref{existence_comb_smoothable}. Koll\'ar \cite[Theorem 16]{Kollar2004Specialisation} has already established a result very close to \Cref{existence_comb_smoothable}, but not exactly in the form required for our purposes. Since the precise form we need does not appear in the literature, we provide complete proofs, following the strategy of \cite{Kollar2004Specialisation} with the modifications required by our setting. Some of our deformation-theoretic arguments over arbitrary fields are similar to those developed in \cite{AraujoKollar2003}.

To construct teeth contained in the fibers of $X \to D$, we study in \Cref{sec:2} free rational curves in the fibers of $X \to D$. This also allows us to treat the case when $D$ has genus zero without separate arguments. Since a rational curve that is very free inside a fiber of $X \to D$ is not very free when viewed as a curve in the total space $X$ (the tangent directions coming from the base $D$ contribute trivial summands to the pullback of the tangent bundle of $X$), we keep track of almost very free curves and of the number of trivial summands in the splitting type of the restricted tangent bundle. In this way, we relate very free curves of $X\to D$ in the fibers to smoothness properties of the relevant evaluation maps. In \Cref{sec:3}, we recall the relationship between separable rational connectedness and very free rational curves, and use it to produce vertical rational curves in many tangent directions through prescribed points of the fibers of a smooth morphism with separably rationally connected fibers. In \Cref{sec:4}, we first establish \Cref{thm:smoothing}. Then, using the abundance of very free curves in the fibers proved in \Cref{sec:3} and following the strategy of \cite[Theorem 16]{Kollar2004Specialisation}, we construct a comb with the required properties inside $X$ (see \Cref{existence_comb_smoothable}), and apply \Cref{thm:smoothing} to this comb to prove \Cref{thm:add_enough_teeth}.

In \Cref{sec:5}, we use \Cref{thm:add_enough_teeth} to prove \Cref{intro:cor_prop_principal_Weil}. Let $\mc{X}\to C$ be a smooth projective morphism with separably rationally connected fibers, let $s\colon C_\kappa\to \mc{X}_\kappa$ be a section, and let $Z\subset C_\kappa$ be a finite subscheme. Applying \Cref{thm:add_enough_teeth}, we obtain a smoothable comb with handle $C_\kappa$ inside $\mc{X}_\kappa$ which agrees with $s$ along $Z$. Smoothing this comb produces a new section $s'\colon  C_\kappa\to \mc{X}_\kappa$ which agrees with $s$ along $Z$. Interpreting this section as a point of the Weil restriction $\mathrm{Res}_{C/R}(\mc{X})$, the fact that the teeth of the comb are chosen to be very free in the fibers, hence almost very free in the total space, implies that this point is smooth. If $\kappa$ is large, one obtains $s'$ over $\kappa$ itself; in general, one obtains finitely many $s_i'\in\mathrm{Res}_{C/R}(\mc{X})(\kappa_i)$ defined over finite separable extensions $\{\kappa_i/\kappa\}_{i\in I}$ of collectively coprime degrees.

In Sections \ref{sec:6}--\ref{sec:7}, we return to the specific situation of torsors over curves. In \Cref{sec:7}, we apply \Cref{cor_prop_principal_Weil} to $\mc{X} = \overline{E}$ to prove \Cref{thm:large-field} and \Cref{thm:main-average}. Finally, a d\'evissage argument shows that \Cref{thm:main} reduces to the cases where (i) $G$ is semisimple and (ii) $G$ is an isotrivial $C$-torus. Case~(i) is \cite[Theorem 7.2]{Gille2021}, and case~(ii) follows from \Cref{thm:main-average} and a restriction-corestriction argument.

The idea of using comb smoothing on a compactification of a torsor, and then interpreting sections of this compactification as points of a Weil restriction, was first suggested by Jason Starr \cite{starr2023local} in the context of the Drinfeld--Simpson theorem. In that setting one works \'etale-locally on the base, and the fact that a smooth morphism admits sections \'etale-locally allows the geometric argument to be carried out after passing to algebraically closed residue fields. This removes most of the technical difficulties we face in this paper.

\subsection{Notation}
Let $k$ be a field. We write $\cl{k}$ for an algebraic closure of $k$. By definition, a $k$-variety is a separated $k$-scheme of finite type, and a $k$-curve is a $k$-variety of pure dimension $1$. 

Let $X$ be a smooth $k$-variety. We let $\mathcal{T}_{X/k} \coloneqq \underline{\operatorname{Hom}}_{\mathcal{O}_X}(\Omega_{X/k}, \mathcal{O}_X)$ be the tangent sheaf of $X$. More generally, given a morphism of smooth $k$-varieties $X \to Y$, we let $\mathcal{T}_{X/Y} \coloneqq \underline{\operatorname{Hom}}_{\mathcal{O}_X}(\Omega_{X/Y}, \mathcal{O}_X)$ be the relative tangent sheaf of $X$ over $Y$.

\section{Vertical rational curves and evaluation maps}\label{sec:2}

\subsection{Rational curves}

Let $k$ be a field. Recall that, by a theorem of Grothendieck \cite{grothendieck1957classification} (see \cite{hazewinkel1982short} for a simple proof for arbitrary $k$), for every integer $r\geqslant 0$ and every rank $r$ vector bundle $\mc{E}$ on $\mathbb{P}^1_k$ there exists a unique sequence of integers $a_1\leqslant a_2\leqslant \dots \leqslant a_r$ such that 
\begin{equation}\label{eq:grothendieck-p1}\mc{E} \cong \bigoplus_{i=1}^r \mathcal{O}_{\mathbb{P}^1_k}(a_i).\end{equation}
We call the sequence $a_1\leqslant a_2\leqslant \dots \leqslant a_r$ the \emph{splitting type} of $\mc{E}$. We let $\mathfrak{z}(\mc{E})$ be the number of indices $i\in\{1,\dots,r\}$ such that $a_i=0$, that is, the rank of the largest trivial direct summand of $\mc{E}$.  

Observe that $a_i \geqslant -1$ for all $i$ if and only if $H^1(\mathbb{P}^1_k, \mathcal{E}) = 0$, 
and that $a_i \geqslant 0$ for all $i$ if and only if $\mathcal{E}$ is globally generated, that is, if and only if the natural map 
$\mathcal{O}_{\mathbb{P}^1_k} \otimes H^0(\P^1_k, \mathcal{E}) \to \mathcal{E}$ is surjective. 
We say that $\mathcal{E}$ is \emph{almost ample} if $a_i \geqslant 0$ for all $i$ and $a_i \geqslant 1$ for every $i > 1$, and \emph{ample} if $a_i \geqslant 1$ for all $i$ (equivalently, if $\mathcal{E} \otimes \mathcal{O}_{\mathbb{P}^1_k}(-1)$ is globally generated).

A \emph{$k'$-rational curve on $X$} is a $k$-morphism $\mathbb{P}^1_{k'} \to X$, where $k'/k$ is a field extension and $X$ is a $k$-variety. A $k'$-rational curve
$f \colon \mathbb{P}^1_{k'} \to X$ is said to be
\begin{itemize}
\item[\textup{(i)}] \emph{free} if $f^* \mathcal{T}_{X/k}$ is globally generated,
\item[\textup{(ii)}] \emph{almost very free} if $f^* \mathcal{T}_{X/k}$ is almost ample, and
\item[\textup{(iii)}] \emph{very free} if $f^* \mathcal{T}_{X/k}$ is ample.
\end{itemize}

\begin{rmk}\label{rmk:field-extension}
    Let $\mc{E}$ be a vector bundle over $\P^1_k$. For every field extension $k'/k$, the splitting type of $\mc{E}$ is equal to the splitting type of $\mc{E}_{k'}$, and in particular $\mathfrak{z}(\mc{E})=\mathfrak{z}(\mc{E}_{k'})$. Therefore, given a smooth $k$-variety $X$ and field extensions $k\subset k'\subset k''$, a $k'$-rational curve $f\colon \P^1_{k'}\to X$ is free (respectively very free, almost very free) if and only if so is the composite $\P^1_{k''}\to  \P^1_{k'} \xrightarrow{f} X$.
\end{rmk}

\subsection{Free curves on the fibers of a smooth morphism}

Let $d\geqslant 0$ be an integer, let $Y$ be a smooth $k$-variety of pure dimension $d$, and let $p \colon X \to Y$ be a smooth morphism. By \cite[Lemma 02K4]{stacks-project}, we have a short exact sequence of vector bundles on $X$:
\begin{equation}\label{exact_sequence_tangent}
0 \longrightarrow \mathcal{T}_{X/Y} \longrightarrow \mathcal{T}_{X/k} \longrightarrow p^*\mathcal{T}_{Y/k} \longrightarrow 0.
\end{equation}
Let $k'/k$ be a field extension, and let $f\colon\P^1_{k'}\to X$ be a $k'$-rational curve such that the composite morphism $p \circ f$ factors through a $k'$-point $y\in Y(k')$. The pullback of \eqref{exact_sequence_tangent} along $f$ is given by the following exact sequence of vector bundles on $\P^1_{k'}$:
\begin{equation}\label{exact_sequence_tangent_pullback}
0 \longrightarrow f_y^* \mathcal{T}_{X_y/k'} \longrightarrow f^* \mathcal{T}_{X/k} \longrightarrow \mathcal{O}_{\mathbb{P}^1_{k'}}^{\oplus d} \longrightarrow 0.
\end{equation}

\begin{lemma}\label{vf_avf}
Let $k$ be a field, let $d\geqslant 0$ be an integer, let $X$ be a smooth $k$-variety, let $Y$ be a smooth $k$-variety of pure dimension $d$, and let $p\colon X \to Y$ be a smooth morphism. Let $k'/k$ be a field extension, let $f\colon \P^1_{k'} \to X$ be a $k'$-rational curve such that the composite $p \circ f$ factors through a $k'$-point $y\in Y(k')$, let $X_y\coloneqq X\times_Yy$, and consider the $k'$-rational curve $f_y\colon \P^1_{k'}\to X_y$ induced by $f$. Consider the following assertions.
\begin{enumerate}
    \item The $k'$-rational curve $f$ is free.
    \item The $k'$-rational curve $f_y$ is free.
    \item The sequence of vector bundles (\ref{exact_sequence_tangent_pullback}) splits.
    \item We have $\mathfrak{z}(f^*\mc{T}_{X/k})=\mathfrak{z}(f_y^*\mc{T}_{X_y/k'})+d$.
\end{enumerate}
Then (1)$\iff$(2) $\Longrightarrow$ (3) $\Longrightarrow$ (4).
\end{lemma}

\begin{proof}
To simplify notation, we write $\mathcal{O}$ for $\mathcal{O}_{\mathbb{P}^1_{k'}}$.

(1) $\Longrightarrow$ (3). Write $f^* \mathcal{T}_{X/k}=E\oplus \mathcal{O}^{\oplus e}$, where $E$ is the direct sum of $\mc{O}(a_i)$ such that $a_i>0$. For every $a>0$, we have $\on{Hom}(\mc{O}(a),\mc{O})=0$, and hence the restriction of $f^* \mathcal{T}_{X/k}\to \mc{O}^{\oplus d}$ to $E$ is trivial. It follows that the restriction of $f^* \mathcal{T}_{X/k}\to \mc{O}^{\oplus d}$ to $\mc{O}^{\oplus e}$ is surjective. Since every surjection $\mc{O}^{\oplus e}\to \mc{O}^{\oplus d}$ is split, we conclude that (\ref{exact_sequence_tangent_pullback}) splits.

(2) $\Longrightarrow$ (3). If $f_y^* \mathcal{T}_{X_y/k'}$ is globally generated, then $\on{Ext}^1(\mathcal{O}^{\oplus d},f_y^* \mathcal{T}_{X_y/k'})\cong H^1(\P^1_{k'},f_y^* \mathcal{T}_{X_y/k'})^{\oplus d}$ vanishes. It follows that the extension (\ref{exact_sequence_tangent_pullback}) splits.

(1)$\iff$(2). Since either of (1) and (2) implies (3), we may assume that the extension (\ref{exact_sequence_tangent_pullback}) splits. Thus, $f^*\mc{T}_{X/k}= f_y^*\mc{T}_{X_y/k'} \oplus \mc{O}^{\oplus d}$, and the conclusion follows by comparing the splitting types of $f_y^*\mc{T}_{X_y/k'}$ and $f^*\mc{T}_{X/k}$. 

(3) $\Longrightarrow$ (4). This is clear.
\end{proof}

\subsection{Schemes of vertical rational curves}

Let $X$ be a quasi-projective $k$-variety. We write $\mathrm{RC}(X)$ for the functor $\underline{\mathrm{Hom}}(\P^1_k,X)$ of rational curves on $X$, that is, the presheaf on the category of $k$-schemes which associates to a $k$-scheme $S$ the set
\[\mathrm{RC}(X)(S)\coloneqq \on{Hom}_k(\P^1_S,X).\]
By \cite[Theorem 5.23]{fantechi2005fundamental}, the presheaf $\mathrm{RC}(X)$ is representable by a $k$-scheme locally of finite type all of whose connected components are quasi-projective. For a $k$-scheme $S$ and a $k$-morphism $f\colon \P^1_S\to X$, we write $[f]$ for the corresponding $S$-point of $\RC(X)$.

\begin{lemma}\label{hom_smooth_very_free}
Let $k$ be a field, let $X$ be a smooth quasi-projective $k$-variety, let $k'/k$ be a field extension, and let $f \colon \mathbb{P}^1_{k'} \to X$ be a $k'$-rational curve on $X$. Suppose that $f$ is free. Then the corresponding $k'$-point $[f] \in \RC(X)(k')$ lies in the smooth locus of $\RC(X)$.
\end{lemma}

\begin{proof}
This follows from \cite[II, Corollary 3.5.4]{kollar1996rational}; see also \cite[Theorem 5.2]{AraujoKollar2003}.
\end{proof}

Let $Y$ be another quasi-projective $k$-variety, and let $p\colon X\to Y$ be a morphism. For a $k$-scheme $S$, a morphism $\P^1_S\to Y$ is said to be \emph{constant} if it factors through the projection $\P^1_S\to S$, that is, if it belongs to the image of the natural map $Y(S)\to Y(\P^1_S)$. We define the subfunctor $\RC(X/Y)$ of $\RC(X)$ as the functor sending a $k$-scheme $S$ to 
\[\RC(X/Y)(S)\coloneqq \{f\colon \P^1_S\to X\;|\; \text{$p\circ f$ is constant}\}.\]
We have a cartesian square of functors
\begin{equation}\label{eq:hom-const-representable}
\begin{tikzcd}
    \RC(X/Y) \arrow[r,"p\circ(-)"]\arrow[d,hook,"\iota"] & Y \arrow[d,hook,"c"] \\
    \RC(X) \arrow[r,"p\circ(-)"] & \RC(Y), 
\end{tikzcd}
\end{equation}
where $\iota$ is the natural inclusion, and where $c$ is the \enquote{constant-maps section}, that is, the morphism induced by the structure morphism $\P^1_k\to\on{Spec}(k)$. The morphism $c$ is a closed immersion, because it is a section of the evaluation morphism $\RC(Y)\to Y$ at any given $k$-point of $\P^1_k(k)$. It is also an open immersion. Indeed, letting $L$ be a very ample line bundle on $Y$, for every $k$-scheme $S$ the function $s\mapsto \on{deg}(f_s^*L)$ is locally constant on $S$; in particular, the locus of points $s\in S$ such that $\on{deg}(f_s^*L)=0$ is open and closed in $S$. It follows that $\RC(X/Y)$ is represented by a $k$-scheme locally of finite type and that $\iota$ is an open and closed immersion.

We define the functor $\RCE(X/Y)$ of rational curve embeddings in $X$ as the subfunctor of $\RC(X/Y)$ which to a $k$-scheme $S$ associates
\[
\RCE(X/Y)(S)
\coloneqq \{ f \colon \mathbb{P}^1_S \to X \mid p\, \circ f \text{ is constant and the induced map }
\P^1_S\to X_S \text{ is a closed immersion} \}.
\]
By \cite[Lemma 05XA]{stacks-project}, this is an open
subscheme of $\RC(X)$. By the next lemma, for every $k$-scheme $S$ we have
\[\RC(X/Y)(S)= \{f\colon \P^1_S\to X\,|\, \text{$(p\circ f)|_{\P^1_{k(s)}}$ is constant for every $s\in S$}\}.\]

\begin{lemma}\label{lem:all-p1-are-constant}
    Let $Y$ be a $k$-variety, let $S$ be a $k$-scheme, and let $f \colon \P^1_S \to Y$ be a $k$-morphism such that for every $s\in S$, the restriction of $f$ to $\P^1_{k(s)}$ is constant. Then $f$ is constant. 
\end{lemma}

\begin{proof}
    Let $p\colon \P^1_S\to S$ be the projection, let $\sigma\colon S\to \P^1_S$ be a section of $p$, and let $f\colon \P^1_S\to Y$ be a morphism. It is enough to show that $f\circ \sigma \circ p=f$. Let $U\subset Y$ be an affine open subscheme. Then the open subscheme $f^{-1}(U)$ of $\P^1_S$ is a union of fibers of $p$, and hence, the projection $p$ being open, we have $f^{-1}(U)=p^{-1}(W)=\P^1_W$, where $W\coloneqq p(f^{-1}(U))$ is an open subscheme of $S$. Since $U$ is affine and $(p|_{\mathbb P^1_W})_*\mc{O}_{\mathbb P^1_W}=\mc{O}_W$, we deduce that the restriction $f|_{f^{-1}(U)}\colon f^{-1}(U)\to U$ factors through a morphism $g\colon W\to U$. This implies that $f\circ \sigma \circ p= g\circ p\circ \sigma\circ p= g\circ p =f$ over $f^{-1}(U)$. By letting $U$ vary over an affine open covering of $Y$, we conclude that $f\circ \sigma \circ p=f$, as desired. 
\end{proof}

\subsection{Evaluation maps} 

Let $X$ be a quasi-projective $k$-variety. We have a $k$-scheme morphism
\begin{equation}\label{eq:evaluation-map}
    \on{ev}\colon\P^1_k\times_k\RC(X)\longrightarrow X,\qquad (t,f)\mapsto f(t),
\end{equation}
called the \emph{evaluation map}, and an induced morphism
\begin{equation}\label{eq:evaluation-map-2}
\operatorname{ev}^{(2)} \colon \P^1_k \times_k \P^1_k \times_k \RC(X) \longrightarrow X \times_k X,\qquad (t,t',f)\mapsto(\on{ev}(t,f),\on{ev}(t',f))=(f(t),f(t')).
\end{equation}
Let $Y$ be another quasi-projective $k$-variety, and let $p\colon X\to Y$ be a morphism. The natural morphism $X\times_YX\to X\times_kX$ is a closed immersion, and the restriction of $\operatorname{ev}^{(2)}$ to $\mathbb{P}^1_k \times_k \mathbb{P}^1_k \times_k\RC(X/Y)$ uniquely factors through a morphism 
\begin{equation}\label{evaluation-map-2-y}
\on{ev}^{(2)}_Y \colon \P^1_k \times_k \P^1_k \times_k \RC(X/Y) \longrightarrow X \times_Y X.
\end{equation}
In other words, we have a commutative square
\[
\begin{tikzcd}
    \P^1_k \times_k \P^1_k \times_k \RC(X/Y) \arrow[r,"\on{ev}^{(2)}_Y"]\arrow[d,hook,"\on{id}\times_k\on{id}\times_k\iota"]  & X\times_YX \arrow[d,hook] \\
    \P^1_k \times_k \P^1_k \times_k \RC(X) \arrow[r,"\on{ev}^{(2)}"] & X\times_kX,
\end{tikzcd}
\]
where $\iota$ is the closed immersion of \eqref{eq:hom-const-representable}. In the special case $Y=\on{Spec}(k)$, we have $\RC(X/Y)=\RC(X)$ and $\on{ev}^{(2)}_Y=\on{ev}^{(2)}$.

\begin{prop}\label{ev_smooth_almost_free}
Let $k$ be a field, let $d\geqslant 0$ be an integer, let $X$ be a smooth quasi-projective $k$-variety, let $Y$ be a smooth $k$-variety of pure dimension $d$, and let $p\colon X\to Y$ be a smooth morphism. Let $k'/k$ be a field extension, let $t,t'\in \P^1_k(k')$ be distinct $k'$-points, let $f\colon \P^1_{k'}\to X$ be a $k'$-rational curve such that $p \circ f$ is constant, let $y\coloneqq p \circ f(t)= p \circ f(t')\in Y(k')$, and let $X_y\coloneqq X\times_Yy$. The following are equivalent.
\begin{enumerate}
\item The $k'$-point $(t,t', [f])$ of $\P^1_k \times_k \P^1_k \times_k \RC(X/Y)$ lies in the smooth locus of $\on{ev}^{(2)}_Y$.
\item The vector bundle $f^*\mc{T}_{X/k}$ is globally generated and $\mathfrak{z}(f^*\mc{T}_{X/k})=d$.
\item The $k'$-rational curve $f_y\colon \P^1_{k'}\to X_y$ is very free.
\end{enumerate}
\end{prop}

\begin{proof}
To simplify notation, we write $\mathcal{O}$ for $\mathcal{O}_{\mathbb{P}^1_{k'}}$. When $Y=\on{Spec}(k)$, \Cref{ev_smooth_almost_free} is proved in \cite[Chapter II, Proposition 3.5]{kollar1996rational}; here we follow the same proof strategy.

By \Cref{rmk:field-extension} and the compatibility of the smooth locus of a morphism under base change, we may assume that $k'=k$. We let $y \in Y(k)$ denote the image of $\P^1_k$ under $p \circ f$. The differential of $\on{ev}^{(2)}_Y$ at $(t,t',f)$ is given by
\begin{equation}\label{tangent_map_C}
\begin{array}{cccc}
T_{(t,t',[f])}(\operatorname{ev}^{(2)}_Y) \colon 
& 
T_{\P^1_k,t} \oplus T_{\P^1_k,t'} \oplus H^0(\P^1_k, f^*\mathcal{T}_{X/k}) 
& \longrightarrow 
& T_{X,f(t)} \times_{T_{Y,y}} T_{X,f(t')} \\
& (u,u',\sigma) 
& \longmapsto 
& \bigl((T_t f)(u) + \sigma(t),\; (T_{t'}f)(u') + \sigma(t')\bigr)
\end{array}
\end{equation}
since the map $\iota$ of \eqref{eq:hom-const-representable} is an open and closed immersion.

(2)$\iff$(3). This follows from \Cref{vf_avf}.

(1) $\Longrightarrow$ (2). By assumption, the map (\ref{tangent_map_C}) is surjective. Observe that $T_{X,f(t)} = (f^*\mathcal{T}_{X/k})_t$ and $T_{\P^1_k,t} = (\mathcal{T}_{\P^1_k/k})_t$. Moreover, we have a commutative diagram
\[
\begin{tikzcd}
H^0(\P^1_k, \mathcal{T}_{\P^1_k/k}) \arrow[r] \arrow[d] 
& T_{\P^1_k,t} \arrow[d,"T_t f"] \\
H^0(\P^1_k, f^*\mathcal{T}_{X/k}) \arrow[r] 
& T_{X,f(t)}
\end{tikzcd}
\]
where the horizontal maps are given by evaluation and the vertical maps are induced by $f$. It follows that the natural map
$H^0(\P^1_k, f^*\mathcal{T}_{X/k}) 
\to 
T_{X,f(t)} \times_{T_{Y,y}} T_{X,f(t')}$
is also surjective. Since $X \to Y$ is smooth, the map $T_{X,f(t')} \to T_{Y,y}$ is surjective, hence $f^*\mathcal{T}_{X/k}$ is globally generated, that is, $f$ is free. Let $X_y \coloneqq X \times_Y y$. By \Cref{vf_avf}, the induced $k$-rational curve 
$f_y \colon \P^1_k \to X_y$
is free and (\ref{exact_sequence_tangent_pullback}) splits. We fix a splitting $f^*\mathcal{T}_{X/k}=(\oplus_{i>d} \mathcal{O}(a_i)) \oplus \mathcal{O}^{\oplus d}$ of (\ref{exact_sequence_tangent_pullback}). Under this identification, the surjection $f^*\mathcal{T}_{X/k}\to \mathcal{O}^{\oplus d}$ is given by the projection onto the last summand. This yields an identification of the map $T_{f(t)}p\colon T_{X,f(t)}\to T_{Y,y}$ with the projection onto the last factor $(\oplus_{i>d} \mathcal{O}(a_i)_t) \oplus k^{\oplus d} \to k^{\oplus d}$. The surjectivity of the map (\ref{tangent_map_C}) implies, for all $i>d$, the surjectivity of the maps $H^0(\mc{O}(a_i))\to \mc{O}(a_i)_t\oplus \mc{O}(a_i)_{t'}$. Since $t\neq t'$, this implies that $a_i\geqslant 1$ for all $i>d$, and hence that $\mathfrak{z}(f^*\mc{T}_{X/k})=d$.

(2) $\Longrightarrow$ (1). By \Cref{hom_smooth_very_free} and the fact that $\RC(X/Y)$ is an open subscheme of $\RC(X)$, the $k$-scheme $\P^1_k\times_k\P^1_k\times_k\RC(X/Y)$ is smooth at $(t,t',[f])$. Again by \Cref{vf_avf}, the extension (\ref{exact_sequence_tangent_pullback}) splits, and we can write
$f^*\mathcal{T}_{X/k} = \bigoplus_{i>d} \mathcal{O}(a_i) \oplus \mathcal{O}^{\oplus d}$, such that $a_i\geqslant 1$ for all $i>d$, so that the pullback of the projection $\mathcal{T}_{X/k} \to p^*\mathcal{T}_{Y/k}$ corresponds to the projection onto the last factor. As the evaluation maps
$H^0(\mathbb{P}^1_k, \mathcal{O}(a_i)) \to \mathcal{O}(a_i)_t \oplus \mathcal{O}(a_i)_{t'}$ 
are surjective for all $i>d$, it follows that the tangent map (\ref{tangent_map_C}) is surjective, and hence that $\mathrm{ev}^{(2)}_Y$ is smooth at $(t,t',[f])$.
\end{proof}

\subsection{Schemes of free curves}

Let $d\geqslant 0$, let $X$ be a smooth quasi-projective $k$-variety,
let $Y$ be a $k$-variety, and let $p\colon X\to Y$ be a morphism. We denote by $\RC(X/Y)_d$ the subfunctor of $\RC(X/Y)$ which associates to a $k$-scheme $S$ the set
\begin{equation}\label{eq:d}
\begin{aligned}
\RC(X/Y)_d(S) 
\coloneqq  \bigl\{[f] \in \RC(X/Y)(S)  \bigm|\; \text{$f_s$ is free and $\mathfrak{z}((f_s)^*\mc{T}_{X/k})\leqslant d$ for all $s\in S$}\bigr\}. \\
\end{aligned}
\end{equation}
We set 
\[\RC(X/Y)_{\mathrm{vf}}\coloneqq \RC(X/Y)_0, \qquad \RC(X/Y)_{\mathrm{avf}}\coloneqq \RC(X/Y)_1.\] Thus, for every $k$-scheme $S$:
\begin{equation}\label{eq:vf}
\RC(X/Y)_{\mathrm{vf}}(S) \coloneqq 
\left\{
[f] \in \RC(X/Y)(S) \;\middle|\; f_s \text{ is very free for all } s \in S
\right\},
\end{equation}
\begin{equation}\label{eq:avf}
\RC(X/Y)_{\mathrm{avf}}(S) \coloneqq 
\left\{
[f] \in \RC(X/Y)(S)\;\middle|\; f_s \text{ is almost very free for all } s \in S\right\}.
\end{equation}
It follows from \Cref{rmk:field-extension} that \eqref{eq:d}-\eqref{eq:avf} determine well-defined subfunctors of $\RC(X/Y)$.

\begin{lemma}\label{smoothness_almost_very_free_locus} Let $k$ be a field, let $d\geqslant 0$ be an integer, let $X$ be a smooth quasi-projective $k$-variety, let $Y$ be a smooth $k$-variety of pure dimension $d$, and let $p \colon X \to Y$ be a smooth morphism. 
The functor $\RC(X/Y)_d$ is represented by a smooth open subscheme of $\RC(X/Y)$.
Moreover, the map $\RC(X/Y)_d \to Y$ is smooth.
\end{lemma}

\begin{proof}
Let $V\subset \P^1_k\times_k\P^1_k$ be the complement of the diagonal, let $W\subset V\times_k\RC(X/Y)$ be the smooth locus of the restriction of the map $\mathrm{ev}_Y^{(2)}$ of (\ref{evaluation-map-2-y}) to $V \times_k \RC(X/Y)$, let $q \colon V \times_k \RC(X/Y) \to \RC(X/Y)$ be the second projection, and define $U_d\coloneqq q(W)\subset  \RC(X/Y)$. Since $q$ is flat, $U_d$ is open in $\RC(X/Y)$. The composite 
\[W\xlongrightarrow{\on{ev}^{(2)}_Y|_W}X\times_YX \longrightarrow Y\]
is smooth and factors through the morphism $U_d\to Y$. Since the morphism $W\to U_d$ is fppf, it follows that the morphism $U_d\to Y$ is smooth. In particular, $U_d$ is smooth over $k$.

It remains to prove that $U_d$ represents $\RC(X/Y)_d$. For this, let $S$ be a $k$-scheme. By \Cref{ev_smooth_almost_free} and the fact that $\dim(Y)=d$, a morphism 
$\varphi \colon S \to \RC(X/Y)$ belongs to $\RC(X/Y)_{d}$ if and only if it set-theoretically factors through $U_d$, that is, since $U_d$ is open in $\RC(X/Y)$, if and only if $\varphi$ scheme-theoretically factors through $U_d$. Therefore $U_d$ represents $\RC(X/Y)_d$, as desired.
\end{proof}

\section{Vertical teeth in separably rationally connected fibrations}\label{sec:3}

\subsection{Separably rationally connected varieties}

Let $k$ be a field, let $\cl{k}$ be an algebraic closure of $k$, and let $X$ be a geometrically integral $k$-variety. Following \cite[IV, (3.1.2) and (3.2.3)]{kollar1996rational}, we say that $X$ is \emph{separably unirational} if there exists an integer $n\geqslant 0$ and a generically smooth dominant rational map $\mathbb{P}^n_k \dashrightarrow X$, and that $X$ is \emph{separably rationally connected}  if there exists an integral $k$-variety $M$ and a morphism
$e\colon \mathbb{P}^1_k \times_k M \to X$
such that the morphism
\begin{equation}\label{eq:src-e-2}
    e^{(2)}\colon \mathbb{P}^1_k \times_k \mathbb{P}^1_k \times_k M \longrightarrow X \times_k X,\qquad (t,t',m)\mapsto (e(t,m),e(t',m))
\end{equation}
is dominant and generically smooth.

The following characterization of separably rationally connected varieties is well known, at least when $k$ is algebraically closed.

\begin{prop}\label{very_free_src}
Let $k$ be a field, and let $X$ be a smooth projective geometrically integral $k$-variety of positive dimension. The following are equivalent.
\begin{enumerate}
    \item The $k$-variety $X$ is separably rationally connected.
    \item There exists a field extension $k'/k$ and a very free rational curve 
$f \colon \mathbb{P}^1_{k'} \to X$.
    \item There exists a finite field extension $k'/k$ and a very free rational curve 
$f \colon \mathbb{P}^1_{k'} \to X$.
\end{enumerate}
\end{prop}

When $k$ is algebraically closed, this is proved in \cite[IV, Theorem 3.7]{kollar1996rational}.

\begin{proof}
(1) $\Longrightarrow$ (3). We first show that the $\cl{k}$-variety $X_{\cl{k}}$ is separably rationally connected. Let $M$ be an integral $k$-variety and $e\colon \P^1_k\times_kM\to X$ be a morphism such that the corresponding morphism $e^{(2)}$ of \eqref{eq:src-e-2} is dominant and generically smooth. In particular, $M$ contains a smooth integral dense open subscheme $U\subset M$. Replacing $M$ by $U$, we may assume that $M$ is smooth. Let $M'$ be one of the irreducible components of the smooth $\cl{k}$-scheme $M_{\cl{k}}$, and let $e'\colon \P^1_{\cl{k}}\times_{\cl{k}}M'\to X_{\cl{k}}$ be the restriction of $e_{\cl{k}}$. Then $M'$ is integral and the corresponding morphism $(e')^{(2)}$ is dominant and generically smooth. Thus $X_{\cl{k}}$ is separably rationally connected, as desired. By \cite[IV, Theorem 3.7]{kollar1996rational}, there exists a very free $\cl{k}$-rational curve $f \colon \mathbb{P}^1_{\cl{k}} \to X$. This morphism descends to a finite subextension $k'/k$.

(3) $\Longrightarrow$ (2). This is clear.

(2) $\Longrightarrow$ (1). We follow the proof of \cite[IV, Theorem~3.7]{kollar1996rational}. Let $k'/k$ be a field extension, and let $f \colon \P^1_{k'} \to X$ be a very free $k'$-rational curve. By \Cref{hom_smooth_very_free}, there exists a smooth integral open subscheme $U \subset \RC(X)$ such that  $[f]\in U(k')$. 
For any two distinct $k'$-points $t,t'\in \P^1_k(k')$, by \Cref{ev_smooth_almost_free} (with $Y=\on{Spec}(k)$) the morphism $\on{ev}^{(2)}_{\Spec k}=\on{ev}^{(2)}$ of (\ref{evaluation-map-2-y}) is smooth at the image of $(t,t',[f])\in (\P^1_k \times \P^1_k \times_k U)(k')$. It follows that the restriction of $\operatorname{ev}^{(2)}$ to $\P^1_k \times \P^1_k \times U$ is generically smooth and dominant. Therefore $X$ is separably rationally connected.
\end{proof}

\begin{cor}\label{src_invariant_base_change}
Let $k$ be a field, let $k'/k$ be a field extension, and let $X$ be a smooth projective geometrically integral $k$-variety. The $k$-variety $X$ is separably rationally connected if and only if so is $X_{k'}$.
\end{cor}
\begin{proof}
This is clear when $\on{dim}(X)=0$, and it is a consequence of \Cref{very_free_src} and \Cref{rmk:field-extension} when $\on{dim}(X)>0$.
\end{proof}

\begin{rmk}\label{unirational_src}
    If $X$ is a separably unirational projective $k$-variety, then $X$ is separably rationally connected; see \cite[IV, Example (3.2.6.2)]{kollar1996rational}. Indeed, let $f\colon \P^n_k\dashrightarrow X$ be a generically smooth dominant rational map, let $e$, $e^{(2)}$ and $M$ be as in \eqref{eq:src-e-2} for $\P^n_k$. Since $e^{(2)}$ is generically smooth, shrinking $M$ if necessary, we may assume that $M$ is smooth. The composite \[(f\circ e)^{(2)}=(f\times f)\circ e^{(2)}\colon \P^1_k\times_k\P^1_k\times_kM\longrightarrow \P^n_k\times_k\P^n_k\dashrightarrow X\times_kX\] is a well-defined generically smooth dominant rational map. It remains to show that, shrinking $M$ if necessary, the composite rational map $f\circ e\colon \P^1_k\times_kM\dashrightarrow X$ is a morphism. Since $X$ is projective, the rational map $f\circ e$ is defined away from a closed subset $Z\subset \P^1_k\times_kM$ of codimension $\geqslant 2$. It follows that $\on{dim}(Z)<\on{dim}(M)$, and hence the image of $Z$ under the second projection $p\colon \P^1_k\times_kM\to M$ is a projective closed subset of $M$. Replacing $M$ by the dense open subscheme $M\setminus p(Z)$, we may assume that $f\circ e$ is a morphism, as desired.
\end{rmk}

\subsection{The tangent evaluation morphism}

Let $k$ be a field, let $p\colon X\to Y$ be a smooth morphism of $k$-varieties, and let $d\geqslant 0$. For every $k$-scheme $S$ and every $[f]\in \RC(X/Y)_d(S)$, let
$0_S\colon S\to \P^1_S$ be the section induced by $0\in \P^1_k(k)$, and set $x\coloneqq f\circ 0_S\colon S\to X$. Let \[D_{0,S}=\Spec_S(\mathcal{O}_S[t]/(t^2))\subset \P^1_S\] be the first-order relative infinitesimal neighbourhood of $0_S$. Since $[f]\in \RC(X/Y)_d(S)$, the composite $p\circ f$ is constant, and hence the restriction $f|_{D_{0,S}}$ defines a relative first-order infinitesimal deformation of $x$ over $Y$, and hence a section
\[
    \mathcal{T}_{0,Y}^{d,X}([f])
    \in \Gamma(S,x^*\mathcal{T}_{X/Y}),
\]
or, equivalently, an $S$-point of $T_{X/Y}$ lying over $x$. This construction is compatible with base change in $S$, and hence defines a
$k$-morphism
\[
    \mathcal{T}_{0,Y}^{d,X}\colon \RC(X/Y)_d\longrightarrow T_{X/Y}
\]
fitting into the commutative triangle
\[
\begin{tikzcd}
\RC(X/Y)_d \arrow[rr, "\mathcal{T}_{0,Y}^{d,X}"] \arrow[dr, "\operatorname{ev}_{0,Y}^d"'] 
&& T_{X/Y} \arrow[dl, "\pi"] \\
& X
\end{tikzcd}
\]
where $\pi\colon T_{X/Y}\to X$ is the projection map.
For every field extension $k'/k$ and every $[f]\in \RC(X/Y)_d(k')$, the morphism $\mathcal{T}_{0,Y}^{d,X}$ sends $[f]$ to its tangent vector at
$0$, denoted $T_0(f)(\partial/\partial t)$.

\begin{lemma}\label{lem:ev-surj}
Let $k$ be a field, let $d \geqslant 0$ be an integer, let $Y$ be a smooth $k$-variety of pure dimension $d$. Let $p \colon X \to Y$ be a smooth projective morphism such that for every $y \in Y$ the $k(y)$-variety $X_y\coloneqq X\times_Yy$ is separably rationally connected. 

\begin{enumerate}
\item  The evaluation tangent map 
\[
    \mathcal{T}_{0,Y}^{d,X} \colon \RC(X/Y)_d \longrightarrow T_{X/Y},
    \qquad [f] \longmapsto T_{0}(f)(\partial/\partial t)
\]
is smooth and surjective. 

\item The evaluation tangent map 
\[
    \mathcal{T}_{0,Y}^{d, X\times_k \P^3_k}|_{\RCE( X \times_k \P^3_k/Y)_d} \colon \RCE( X \times_k \P^3_k/Y)_d \longrightarrow T_{X \times_k \P^3_k/Y},
    \qquad [f] \longmapsto T_{0}(f)(\partial/\partial t)
\]
is smooth and its image contains $T_{X \times_k \P^3_k/Y}\setminus (p_1^*T_{X/Y}\oplus \{0\})$, where $p_1\colon X \times_k \P^3_k\to X$ denotes the first projection.
\end{enumerate}
\end{lemma}

\begin{proof}
We may assume that $k$ is algebraically closed.

(1) Since the morphisms $\RC(X/Y)_d \to Y$ and $T_{X/Y} \to Y$ are smooth, the smoothness of $\mathcal{T}_{0,Y}^{d,X}$ can be checked by verifying the surjectivity of its differential at each $k$-point of $\RC(X/Y)_d$. This may be done fiberwise over $Y$, so that we may reduce to the case $Y = \operatorname{Spec} k$. In other words, it suffices to show that if $X$ is separably rationally connected over $k$, then the morphism
\[
    \mathcal{T}_{0,\operatorname{Spec} k}^{0,X} \colon \RC(X)_{\mathrm{vf}} \longrightarrow T_{X/\operatorname{Spec} k},
    \quad [f] \longmapsto T_{0}(f)(\partial/\partial t),
\]
has surjective differential at every $k$-point $[f]\in \RC(X)_{\mathrm{vf}}(k)$. This is equivalent to the surjectivity of the map
\begin{equation}\label{eq:exact_sequence_O_2}
    H^0\!\left(\mathbb{P}^1_k,\, f^*\mathcal{T}_{X/k}\right) \longrightarrow f^*\mathcal{T}_{X/k} \otimes \mathcal{O}_{\mathbb{P}^1_k, 0} / \mathfrak{m}_{\mathbb{P}^1_k, 0}^2.
\end{equation}
Consider the exact sequence
\[
    0 \longrightarrow \mathcal{O}_{\mathbb{P}^1_k}(-2 \cdot [0]) \otimes f^*\mathcal{T}_{X/k} \longrightarrow f^*\mathcal{T}_{X/k} \longrightarrow \left(\mathcal{O}_{\mathbb{P}^1_k, 0} / \mathfrak{m}_{\mathbb{P}^1_k, 0}^2\right) \otimes f^*\mathcal{T}_{X/k} \longrightarrow 0.
\]
Since $f$ is very free, $H^1(\mathbb{P}^1_k,\, \mathcal{O}_{\mathbb{P}^1_k}(-2 \cdot [0]) \otimes f^*\mathcal{T}_{X/k}) = 0$. The associated long exact sequence in cohomology then shows that the map~\eqref{eq:exact_sequence_O_2} is surjective. Hence the differential of $\mathcal{T}_{0,\operatorname{Spec} k}^{0,X}$ is surjective at every point, and therefore $\mathcal{T}_{0,\operatorname{Spec} k}^{0,X}$ is smooth. Surjectivity follows from \cite[Theorem 2.42]{Debarre2010GAeL}.

(2) As $\RCE(X \times_k \mathbb{P}^3_k/Y)_d$ is an open subscheme of 
$\RC(X \times_k \mathbb{P}^3_k/Y)_d$, 
the smoothness of the map is implied by (1); it is then enough to show that its image contains $T_{X \times_k \P^3_k/Y}\setminus (T_{X/Y}\oplus \{0\})$. We can then suppose $k$ algebraically closed. Since $\mathcal{T}_{X \times_k \mathbb{P}^3_k/Y} = p_1^* \mathcal{T}_{X/Y} \oplus p_2^* \mathcal{T}_{\mathbb{P}^3_k/k}$, it is enough to prove that the image of
$\mathcal{T}_{0,\Spec k}^{0,\P^3_k} \colon \RCE(\mathbb{P}^3_k/ \Spec k)_{\mathrm{vf}} 
\to T_{\mathbb{P}^3_k/k}$ contains $T_{\mathbb{P}^3_k/k}\setminus\{0\}$, that is, for every $k$-point $p$ of $\mathbb{P}^3_k$ and every $u \in T_p \mathbb{P}^3_k \setminus \{0\}$, there exists a line in 
$\mathbb{P}^3_k$ passing through $p$ with tangent direction $u$, which is clear.
\end{proof}

\subsection{Vertical teeth through prescribed points}

The following proposition supplies the vertical rational curves with prescribed first-order behaviour that will be used later, together with comb smoothing, to arrange incidence conditions while preserving almost very freeness.

\begin{prop}\label{avf_passing_Z}
Let $k$ be a field, let $d \geqslant 0$ be an integer, let $Y$ be a smooth 
$k$-variety of pure dimension $d$, let $p \colon X \to Y$ be a smooth projective 
morphism such that for every $y \in Y$ the $k(y)$-variety 
$X_y \coloneqq X \times_Y y$ is separably rationally connected, and let 
$\iota \colon Z \hookrightarrow X$ be a positive-dimensional irreducible closed subscheme of codimension $m\geq 0$ such that $p|_Z \colon Z \to Y$ is smooth.

(1) The set
\begin{align*}
\left\{
z \in Z \;\middle|\;
\begin{aligned}
& k(z)/k \text{ is finite and separable, and there exist vectors }
  v_1, \dots, v_m \in (\mathcal{T}_{X/Y})_z \\
& \text{whose images in } (\mathcal{N}_{Z/X})_z \text{ form a basis, and for every }
  i=1,\dots,m \text{ there exists} \\
& \text{a very free curve } f_i \colon \mathbb{P}^1_{k(z)} \to X_{p(z)} \text{ such that } 
  f_i(0) = z \text{ and } \mathrm{d}_0 f_i\!\left(\tfrac{\partial}{\partial t}\right) = v_i
\end{aligned}
\right\}
\end{align*}
is dense in $Z$.

(2) We denote by $p' \colon X \times_k \mathbb{P}^3_k \to Y$ the composite of the first projection with $p$.  Let $\iota' \colon Z' \hookrightarrow X \times_k \mathbb{P}^3_k$ be a closed immersion, where $Z'$ is a positive-dimensional irreducible closed subscheme of codimension $m'\geq 0$ such that $p'|_{Z'}\colon Z' \to Y$ is smooth.
Then the set
\begin{align*}
\left\{ z' \in Z' \;\middle|\;
\begin{aligned}
& k(z')/k \text{ is finite and separable, and there exist vectors }
v_1, \dots, v_{m'} \in (\mathcal{T}_{X \times_k \mathbb{P}^3_k/Y})_{z'} \\
& \text{whose images in } (\mathcal{N}_{Z'/X \times_k \mathbb{P}^3_k})_{z'}
 \text{ form a basis, and for every } i=1,\dots,m' \text{ there exists an}\\
& \text{embedded very free curve }
f_i \colon \mathbb{P}^1_{k(z')} \to (X \times_k \mathbb{P}^3_k)_{p'(z')}
\text{ such that } f_i(0) = z' \text{ and }
\mathrm{d}_0 f_i\!\left(\tfrac{\partial}{\partial t}\right) = v_i
\end{aligned}
\right\}
\end{align*}
is dense in $Z'$.
\end{prop}

\begin{proof} 
(1) We follow the idea of \cite[p. 8]{Kollar2004Specialisation}.
First, as $p|_Z \colon Z \to Y$ is smooth, by \cite[Lemma 067T]{stacks-project}, the closed immersion $\iota\colon Z\hookrightarrow X$ is a regular immersion. We then have the following exact sequence
\[
0 \longrightarrow \mathcal{T}_{Z/Y} \longrightarrow
\iota^*\mathcal{T}_{X/Y} \longrightarrow \mathcal{N}_{Z/X}
\longrightarrow 0.
\]

Let $v_1, \dots, v_n$ be a basis of the $k(Z)$-vector space $(\mathcal{T}_{X/Y}|_Z)_{k(Z)}$ such that the images of the first $m$ vectors $\overline{v}_1, \dots, \overline{v}_m$ in $(\mathcal{N}_{Z/X})_{k(Z)}$ form a basis. For each $i=1,\dots,n$, let $N(Z, v_i)$ be the closure in $\mathcal{T}_{X/Y} \times_X Z$, and thus in $\mathcal{T}_{X/Y}$, of the orbit $\mathbb{G}_{m,k(Z)}\cdot v_i\subset \mathcal{T}_{X/Y}$, and define $H(Z,v_i)$ by the cartesian square
\[
\begin{tikzcd}
H(Z,v_i) \arrow[r] \arrow[d]
& \RC(X/Y)_d \arrow[d, "\mathcal{T}_{0,Y}^{d,X}"] \\
N(Z,v_i) \arrow[r, hook]
& \mathcal{T}_{X/Y},
\end{tikzcd}
\]
thus by \Cref{lem:ev-surj}(1), the map $H(Z,v_i) \to N(Z,v_i)$ is smooth and surjective.
Over a dense open subset $U$ of $Z$ over which $v_1,\dots,v_n$ extend to a basis of $\mathcal{T}_{X/Y}|_U$, for every $i=1,\dots,n$, the restriction $N(Z,v_i)|_U$ is the line subbundle generated by $v_i$. Set
\[
N(Z,v_i)^\times\coloneqq N(Z,v_i)|_U\setminus 0_U
\qquad\text{and}\qquad
H(Z,v_i)^\times\coloneqq H(Z,v_i)\times_{N(Z,v_i)}N(Z,v_i)^\times.
\]
Then $N(Z,v_i)^\times\simeq \mathbb G_{m,U}$, and the morphism
$H(Z,v_i)^\times\to N(Z,v_i)^\times$ is smooth and surjective. Hence the natural morphism
$H(Z,v_1)^\times\times_U\cdots\times_U H(Z,v_n)^\times\to U$
is identified with the composite
\[
H(Z,v_1)^\times\times_U\cdots\times_U H(Z,v_n)^\times
\longrightarrow
N(Z,v_1)^\times\times_U\cdots\times_U N(Z,v_n)^\times\xlongrightarrow{\sim} \mathbb{G}_{m,U}^n\longrightarrow U,
\]
and hence in particular is smooth and surjective over $U$.
Thus the set of closed points $z \in U$ such that $k(z)$ is separable over $k$
and for which there exists a closed point $\eta$ in $H(Z,v_1)^\times\times_U\cdots\times_U H(Z,v_n)^\times$
mapping to $z$ with
$k(\eta) = k(z)$ is dense in $U$, and hence in $Z$; see \cite[Proposition~6.1]{Gille2021}.
Fix such $z \in U$ and $\eta$. The point $\eta$ corresponds to a family of curves
$f_i \colon \mathbb{P}^1_{k(z)} \to X$ such that $p \circ f_i$ is constant,
$f_i(0) = z$, $(f_i)_{p(z)} \colon \P^1_{k(z)} \to X_{k(p(z))}$ is very free, and, for every $i$, $T_0(f_i)(1)$ is a nonzero scalar multiple of $v_i(z)$. In particular,
$(T_0(f_i)(1))_{1 \leqslant i \leqslant n}$ spans
$(\mathcal{T}_{X/Y})_z$ and the image $(\overline{T_0(f_i)(1)})_{1 \leqslant i \leqslant m}$ in $(\mathcal{N}_{Z/X})_z$ is a basis. This proves the first assertion.

(2) The proof is the same as the proof of (1), applying \Cref{lem:ev-surj}(2) instead of  \Cref{lem:ev-surj}(1), and fixing a basis $v_1,\dots,v_n$ of $(\mathcal{T}_{X \times_k \P^3_k/Y}|_{Z'})_{k(Z')}=((p_1|_{Z'})^*\mathcal{T}_{X/Y})_{k(Z')}\oplus ((p_2|_{Z'})^*\mathcal{T}_{\P^3_k/k})_{k(Z')}$ such that the images of $v_1,\dots,v_{m'}$ in $(\mathcal{N}_{Z'/X\times_k\P^3_k})_{k(Z')}$ form a basis and such that no $v_i$ lies in the proper subspace $((p_1|_{Z'})^*\mathcal{T}_{X/Y})_{k(Z')}\oplus \{0\}$.
\end{proof}

\section{Comb smoothing of sections}\label{sec:4}

\subsection{Combs and smoothings}

Our proofs of Theorems \ref{thm:main}, \ref{thm:large-field} and \ref{thm:main-average} rely on the comb smoothing technique. We begin by recalling the definitions and results that are relevant to us, following Koll\'ar \cite{kollar1996rational}.

\begin{defin} \label{def:smoothing}
Let $k$ be a field, let $C$ be a curve over $k$, let $X$ be a $k$-variety,
let $f \colon C \to X$ be a morphism, let $\iota \colon Z \hookrightarrow C$ be a finite subscheme of $C$, and let $n \geqslant 0$ be an integer.
\begin{enumerate}
    \item A \emph{smoothing} of $f$ is a commutative diagram of $k$-varieties
    \begin{equation}\label{eq:smoothing}
    \begin{tikzcd}
        C \arrow[d] \arrow[r, hook] \arrow[rr, bend left=30, "f"] 
            & Y \arrow[r, "F"] \arrow[d, "q"] 
            & X \\
        \operatorname{Spec}(k) \arrow[r, hook, "t_0"] 
            & T
    \end{tikzcd}
    \end{equation}
    where $F$ is a morphism, $T$ is a smooth connected $k$-curve, the square is cartesian,
    and the morphism $q$ is flat and projective and restricts to a smooth
    morphism $Y \setminus C \to T \setminus \{t_0\}$.

    \item We say that the smoothing~\eqref{eq:smoothing} \emph{fixes} $f(Z)$
    if there exists a morphism $s \colon T \times_k Z \longrightarrow Y$    such that the following conditions hold:
    \[
        q \circ s = \operatorname{pr}_T,
        \qquad
        s|_{\{t_0\}\times_k Z} = \iota_Y,
        \qquad
        F \circ s = f|_Z \circ \operatorname{pr}_Z.
    \]
    Here $\operatorname{pr}_T \colon T\times_k Z\to T$ and
    $\operatorname{pr}_Z \colon T\times_k Z\to Z$ are the two projections, and
    $\iota_Y\colon Z\hookrightarrow C\hookrightarrow Y$ denotes the natural
    inclusion, after identifying $\{t_0\}\times_k Z$ with $Z$.

    \item The morphism $f \colon C \to X$ is said to be
    \emph{smoothable fixing} $f(Z)$ if there exists a smoothing~\eqref{eq:smoothing}
    which fixes $f(Z)$.
\end{enumerate}
\end{defin}

We recall that, for a field $k$ and a $k$-curve $C$, a closed point $p\in C$ is said to be an \emph{ordinary node} if for every $\cl{k}$-point $q\in C_{\cl{k}}(\cl{k})$ lying over $p$, we have $\hat{\mc{O}}_{C_{\cl{k}},q}\cong \cl{k}[\![x,y]\!]/(xy)$.

\begin{defin}[{\cite[II, Definition 7.7]{kollar1996rational}}]\label{comb}
Let $k$ be a field, and let $n\geqslant 0$ be an integer. A \emph{$k$-comb with $n$ teeth} is a geometrically connected, geometrically reduced projective $k$-curve $\mathcal{C}$, with irreducible components 
$D$ (the \emph{handle} of the comb) and $C_1, \dots, C_n$ (the \emph{teeth} of the comb) satisfying the following properties:
\begin{enumerate}
    \item[\textup{(i)}] $D$ is smooth and each $C_i$ is isomorphic to $\mathbb{P}^1_{k'_i}$ for some finite separable extension $k'_i/k$, 
    \item[\textup{(ii)}] the only singularities of $\mathcal{C}$ are ordinary nodes,
    \item[\textup{(iii)}] the teeth $C_i$ are pairwise disjoint, and each $C_i$ intersects $D$ in a single closed point $x_i$.
\end{enumerate}
A closed subscheme $\mathcal{C}'\subset \mathcal{C}$ is a \emph{subcomb} if it is a union of irreducible components of $\mathcal{C}$ and $D\subset \mathcal{C}'$.
\end{defin}

\begin{rmk}\label{Serre_ampleness}
Let $k$ be a field, and let $D$ be a smooth projective connected curve over $k$.  
An application of Serre's ampleness criterion (see \cite[Theorem 23.6]{GortzWedhornalgebraic}) shows that for every vector bundle $E$ on $D$, there exists an integer $n_E$ such that for every line bundle $\mathcal{L}$ on $D$ of degree $\deg(\mathcal{L}) \geqslant n_E$, one has
$H^1(D, E \otimes \mathcal{L}) = 0$. 
\end{rmk}

For the proofs of Theorems \ref{thm:main}, \ref{thm:large-field} and \ref{thm:main-average}, we will need to consider smoothings of combs whose handles $D$ are sections of a given morphism $X\to D$. We first record that such a morphism is automatically smooth along the handle.

\begin{lemma}\label{smooth_along_section}
Let $k$ be a field, let $D$ and $X$ be smooth $k$-schemes, let $p\colon X\to D$ be a morphism, and let $s\colon D\to X$ be a section of $p$. Then $p$ is smooth in an open neighborhood of $s(D)$.
\end{lemma}

\begin{proof}
The equality $p\circ s=\operatorname{id}_D$ implies that the differential of $p$ is surjective at every point of $s(D)$. Since $X$ and $D$ are smooth over $k$, the differential criterion for smoothness shows that $p$ is smooth at every point of $s(D)$. The conclusion follows from the openness of the smooth locus.
\end{proof}

The next lemma shows that a general fiber of the smoothing is canonically isomorphic to $D$.

\begin{lemma} \label{continuity_smoothing}
Let $k$ be a field, let $D$ be a smooth projective geometrically connected curve over $k$, let $X$ be a smooth projective $k$-variety, and let $p\colon X \to D$ be a morphism. Let $\mathcal{C}$ be a comb with handle $D$ and teeth $(C_i)_{1 \leqslant i \leqslant m'}$, and let $f \colon \mathcal{C} \to X$ be a morphism whose restriction to $D$ is a section of $p$ and such that $p\circ(f|_{C_i})$ is constant for every $1\leqslant i\leqslant m'$.  
Suppose given a smoothing of $f$ as in (\ref{eq:smoothing}). Then there exists a dense open subscheme $U \subset T$ such that for all points $t \in U$, the composite
\[Y_t\xlongrightarrow{F|_{Y_t}} X_{k(t)}\xlongrightarrow{p_{k(t)}}D_{k(t)}\]
is an isomorphism. 
\end{lemma}

\begin{proof}
We have $H^0(\mathcal{C}, \mathcal{O}_{\mathcal{C}}) = k$. It follows from the base change theorem for coherent cohomology \cite[III, Theorem 12.11]{Hartshorne1977algebraic} that there exists an open neighborhood of $t_0$ over which $q_* \mathcal{O}_Y$ is locally free of rank $1$. Replacing $T$ by this neighborhood, we may assume that $Y_t$ is
irreducible for all $t\in T\setminus\{t_0\}$. Let $x_i$ be the unique point in $D\cap C_i$. Since the morphism $q$ is projective and the $k$-curve $\mathcal{C}$ is projective, the morphism $g \coloneqq (p\circ F,q) \colon
Y \to D \times_k T$ is projective. Let $D^\circ \coloneqq (D \times_k T) \setminus \{(x_1,t_0),\dots,(x_{m'},t_0)\}$. Since $g$ maps every tooth $C_i$ to $(x_i,t_0)$, the restriction $g|_{g^{-1}(D^\circ)} \colon g^{-1}(D^\circ) \to D^\circ$ is a morphism between smooth schemes over $T$ which induces an isomorphism between the fibers above $t_0$. By \cite[Lemma 04DI]{stacks-project}, there exists an open neighborhood $U \subset D^\circ$ of $(D\setminus\{x_1,\dots,x_{m'}\}) \times \{t_0\}$ such that $g|_{g^{-1}(U)} \colon g^{-1}(U)
\to U$ has $0$-dimensional fibers. Thus, by the miracle flatness theorem
\cite[Lemma 00R4]{stacks-project}, this morphism is flat. Since $g|_{g^{-1}(U)}$ is projective and flat and induces an isomorphism over $t_0 \in T(k)$, it follows that there exists an open subscheme $U' \subset U$ containing $(D\setminus\{x_1,\dots,x_{m'}\}) \times \{t_0\}$ such that $g|_{g^{-1}(U')}$ is an isomorphism. In particular, there exists an open neighborhood $V \subset T$ of $t_0$ such that
for all $t \in V\setminus \{t_0\}$, the composite $p_{k(t)}\circ F|_{Y_t} \colon Y_t\to D_{k(t)}$ is a birational morphism between smooth projective irreducible curves, and hence an isomorphism. Thus $U\coloneqq V\setminus\{t_0\}$ is the required dense open subscheme of the smooth connected curve $T$.
\end{proof}

\subsection{The deformation theory of embedded combs}

Let $k$ be a field, let $A$ be a $k$-algebra, and let $I\subset A$ be an ideal. Let $\operatorname{Def}_A(I)$ be the set of first-order deformations of $I$, that is, the set of ideals $J\subset A[\epsilon]/(\epsilon^2)$ such that $(A[\epsilon]/(\epsilon^2))/J$ is flat over $k[\epsilon]/(\epsilon^2)$ and $
\frac{J+(\epsilon)}{(\epsilon)}=I$. 

Let $J\in \operatorname{Def}_A(I)$. For every $f\in I$, choose $a\in A$ such that $f-\epsilon a\in J$ and set $\varphi_J(f+I^2)= a+I$. One checks that $\varphi_J\colon I/I^2\to A/I$ is a well-defined homomorphism of $A/I$-modules. We obtain a function
\begin{equation}\label{eq:def-to-hom}
\operatorname{Def}_A(I)
\longrightarrow
\operatorname{Hom}_{A/I}(I/I^2,A/I),\qquad J\mapsto \varphi_J.
\end{equation}
Conversely, given $\varphi\in \operatorname{Hom}_{A/I}(I/I^2,A/I)$, 
define an ideal
\[
J_\varphi=\{f-\epsilon a\in A[\epsilon]/(\epsilon^2)\;|\; f\in I,\ a\in A,\ a+I=\varphi(f+I^2)\}\subset A[\epsilon]/(\epsilon^2).
\]
One checks that $J_\varphi\in \operatorname{Def}_A(I)$. We obtain a function
\begin{equation}\label{eq:hom-to-def}
\operatorname{Hom}_{A/I}(I/I^2,A/I)
\longrightarrow
\operatorname{Def}_A(I),\qquad \varphi\mapsto J_\varphi.
\end{equation}
It is not difficult to show that the two functions \eqref{eq:def-to-hom} and \eqref{eq:hom-to-def} are inverse to each other; see \cite[Proof of Proposition 2.3]{hartshorne2010deformation}. These constructions are compatible with localization of $A$.

Now let $X$ be a $k$-scheme, let $Y\subset X$ be a closed subscheme, with ideal sheaf $\mathcal I\subset \mc{O}_X$. Recall that a first-order embedded deformation of $Y$ in $X$ is a closed subscheme $\mathcal Y\subset X\times_k\Spec(k[\epsilon]/(\epsilon^2))$
which is flat over $k[\epsilon]/(\epsilon^2)$ and whose special fiber is $Y$.
By covering $X$ by affine open subschemes and applying the previous construction, we obtain a bijective correspondence between the set of first-order embedded deformations of $Y$ in $X$ and $\operatorname{Hom}_{\mc{O}_Y}(\mathcal I/\mathcal I^2,\mc{O}_Y)=H^0(Y,N_{Y/X})$; see \cite[Theorem 2.4]{hartshorne2010deformation}.

The following lemma is a well-known result; see for example \cite[Lemma 2.5]{graber2003families}. We include a proof for lack of a suitable reference for the proof.

\begin{lemma}\label{prop:transverse-curve}
Let $k$ be a field, let $X$ be a smooth projective $k$-variety, and let $\mc{C}\subset X$ be a comb with handle $D$ and teeth $C_1,\dots,C_m$. For all $i=1,\dots,m$, let $q_i\in D$ be the unique closed point such that $D\cap C_i=\{q_i\}$, suppose that the finite extension $k(q_i)/k$ is separable, and let
\[
\rho_i\colon H^0(\mc{C},N_{\mc{C}/X})
\longrightarrow
H^0(D,(N_{\mc{C}/X})|_D)
\longrightarrow
T_{C_i,q_i}
\]
be the map induced by the short exact sequence
\begin{equation}\label{normal_exact_sequence}
    0\longrightarrow N_{D/X}
\longrightarrow (N_{\mc{C}/X})|_{D}
\longrightarrow \bigoplus_{i=1}^m T_{C_i,q_i}
\longrightarrow 0.
\end{equation}
Let $v\in T_{[\mc{C}]}\operatorname{Hilb}_X = H^0(\mc{C},N_{\mc{C}/X})$ 
be such that $\rho_i(v)\neq 0$ for every $i=1,\dots,m$. Let $Z\subset \operatorname{Hilb}_X$ be a connected curve which passes through $[\mc{C}]$, is smooth at $[\mc{C}]$, and such that $v$ belongs to $T_{[\mc{C}]}Z\subset T_{[\mc{C}]}\operatorname{Hilb}_X$. Then there exists a dense open subscheme $U\subset Z$ such that every $u\in U(\overline k)$ corresponds to a smooth projective curve in $X$.
\end{lemma}

\begin{proof}
We may assume that $k$ is algebraically closed. Let $s\coloneqq [\mc{C}]\in Z(k)$, and let $\pi\colon\widetilde{\mathcal C}\longrightarrow Z$ be the restriction of the universal family over the Hilbert scheme $\on{Hilb}_X$ to $Z\subset \on{Hilb}_X$. Thus $\widetilde{\mc{C}}$ is a
closed subscheme of $X\times Z$, flat and projective over $Z$, and the
fiber over $s$ is $\mc{C}$. Fix an integer $i\in\{1,\dots,m\}$, and let $q\coloneqq q_i\in \mc{C}(k)$. We may find a regular sequence $x,y,z_1,\dots,z_r$ in ${\mc{O}}_{X,q}$ such that $x,z_1,\dots,z_r$ generate the ideal of $D$ at $q$ and $y,z_1,\dots,z_r$ generate the ideal of $C_i$ at $q$, that is,
\begin{equation}\label{eq:b-local-parameters}
{\mc{O}}_{D,q}
=
{\mc{O}}_{X,q}/(x,z_1,\dots,z_r),\qquad 
{\mc{O}}_{C_i,q}
=
{\mc{O}}_{X,q}/(y,z_1,\dots,z_r),\qquad
{\mc{O}}_{\mc{C},q}
=
{\mc{O}}_{X,q}/(xy,z_1,\dots,z_r).
\end{equation}
By \cite[Lemma 07NM]{stacks-project}, we obtain an identification
\begin{equation}\label{b:local-x}
\widehat{\mc{O}}_{X,q}
= k[[x,y,z_1,\dots,z_r]].
\end{equation}
By \eqref{eq:b-local-parameters} and \eqref{b:local-x}, we obtain
\begin{align}
\widehat{\mc{O}}_{D,q}
=
\widehat{\mc{O}}_{X,q}/(x,z_1,\dots,z_r)&=k[[x,y,z_1,\dots,z_r]]/(x,z_1,\dots,z_r),
\label{eq:b-complete-rings-C0}\\
\widehat{\mc{O}}_{C_i,q}
=
\widehat{\mc{O}}_{X,q}/(y,z_1,\dots,z_r)&=k[[x,y,z_1,\dots,z_r]]/(y,z_1,\dots,z_r),
\label{eq:b-complete-rings-Ci}\\
\widehat{\mc{O}}_{\mc{C},q}
=
\widehat{\mc{O}}_{X,q}/(xy,z_1,\dots,z_r)&=k[[x,y,z_1,\dots,z_r]]/(xy,z_1,\dots,z_r).
\label{eq:b-complete-rings-C}
\end{align}
Fix a local parameter $t$ for $Z$ at $s$ such that 
\begin{equation}\label{eq:b-tangent-vector-is-v}
    \text{the tangent vector $\Spec k[\epsilon]/(\epsilon^2)\longrightarrow Z$ defined by $t\mapsto \epsilon$ maps to $v$ under
$T_sZ\to T_{[\mc{C}]}\operatorname{Hilb}_X$.}
\end{equation} 
We have identifications
\begin{equation}
    \widehat{\mc{O}}_{Z,s} = k[[t]],\qquad \widehat{\mc{O}}_{X\times Z,(q,s)}
= k[[t,x,y,z_1,\dots,z_r]],\qquad \widehat{\mc{O}}_{\widetilde{\mc{C}},(q,s)}
=
k[[t,x,y,z_1,\dots,z_r]]/J,
\end{equation} 
for some ideal $J$. 

Since the fiber of $\widetilde{\mc{C}}$ over $s$ is $\mc{C}$, we have
\begin{equation}\label{eq:b-j+t-t}
\frac{J}{J\cap (t)}=\frac{J+(t)}{(t)}
=
(xy,z_1,\dots,z_r)
\subset k[[x,y,z_1,\dots,z_r]].
\end{equation}
Since $\widehat{\mc{O}}_{\widetilde{\mc{C}},(q,s)}$ is flat over $\widehat{\mathcal{O}}_{Z,s}$, by \cite[Lemma 00HD]{stacks-project} we have $J\cap (t)=tJ$. Therefore the surjection $J\to \frac{J+(t)}{(t)}$ induces an isomorphism $J/tJ\cong \frac{J+(t)}{(t)}$. Choose elements $F_0,F_1,\dots,F_r\in J$ whose images modulo $t$ are $xy,z_1,\dots,z_r$, respectively. Then $J=(F_0,F_1,\dots,F_r)+tJ$, and hence by Nakayama's lemma $J=(F_0,F_1,\dots,F_r)$, that is, $F_0,F_1,\dots,F_r$ generate $J$. For $1\leq j\leq r$, we have $F_j\equiv z_j\pmod t$, and hence the elements
$t,x,y,F_1,\dots,F_r$ form a regular system of parameters of
$k[[t,x,y,z_1,\dots,z_r]]$. Hence, by
\cite[Lemma 07NM]{stacks-project}, we have an isomorphism
\begin{equation}\label{eq:b-txy}
k[[t,x,y,z_1,\dots,z_r]]/(F_1,\dots,F_r)\cong k[[t,x,y]].
\end{equation}
Under the isomorphism \eqref{eq:b-txy}, the remaining power series $F_0$ is sent to some $G=G(t,x,y)\in k[[t,x,y]]$ satisfying $G(0,x,y)=xy$, that is, $
G=xy-tu$ for some $u=u(t,x,y)\in k[[t,x,y]]$. Thus
\begin{equation}\label{eq:b-txyg}
\widehat{\mathcal{O}}_{\widetilde{\mc{C}},(q,s)}\cong k[[t,x,y]]/(G)=k[[t,x,y]]/(xy-tu).
\end{equation}
We have $tu(t,x,y)\equiv tu(0,x,y)\pmod{t^2}$. Therefore, the first-order deformation of $\mc{C}$ at $q$ corresponding to \eqref{eq:b-txyg} is given by the ideal $(xy-t u(0,x,y))
\subset k[[t,x,y]]/(t^2)$.
By the construction of \eqref{eq:def-to-hom}, the homomorphism $(xy)/(xy)^2\longrightarrow k[[x,y]]/(xy)$ corresponding to this deformation sends $xy+(xy)^2$ to $u(0,x,y)+(xy)$. By \eqref{eq:b-tangent-vector-is-v}, this homomorphism is the restriction, near $q=q_i$, of the tangent vector $v\in H^0(\mc{C},N_{\mc{C}/X})$. 

Under the identification of \eqref{b:local-x}, the ideal of $\mc{C}$ is $I_{\mc{C}}=(xy,z_1,\dots,z_r)$, and the ideal of $D$ is $I_{D}=(x,z_1,\dots,z_r)$; see \eqref{eq:b-complete-rings-C} and \eqref{eq:b-complete-rings-C0}. Recall how the sequence \eqref{normal_exact_sequence} is defined. The inclusion $I_{\mc{C}}\subset I_{D}$ induces a homomorphism $I_{\mc{C}}/I_{\mc{C}}^2\otimes_{\mc{O}_{\mc{C}}}\mc{O}_{D}\to I_{D}/I_{D}^2$. 
It sends the class of $xy$ to $y$ times the class of $x$, and it sends
the class of $z_\ell$ to the class of $z_\ell$. Dualizing, we get the map $N_{D/X}
\to (N_{\mc{C}/X})|_{D}$. Now restrict the section $v$ to $D$. Its value on the class of $xy$
is $u(0,0,y)\in \widehat{\mc{O}}_{D,q}$. On the other hand, any section coming from $N_{D/X}$ has value on the class of $xy$ lying in $y\widehat{\mc{O}}_{D,q}$. Therefore the image of $v$ in the cokernel $(N_{\mc{C}/X})|_{D}/N_{D/X}$ 
is represented at $q$ by the class of $u(0,0,y)$ modulo $y\widehat{\mc{O}}_{D,q}$. Since $\widehat{\mc{O}}_{D,q}/y\widehat{\mc{O}}_{D,q}\cong k$,
this class is equal to $u(0,0,0)$. Thus $\rho_i(v)\neq 0$ implies $u(0,0,0)\neq 0$. Therefore $u(t,x,y)$ is a unit in $k[[t,x,y]]$.

Since $\pi$ is flat, the singular locus of $\pi$ coincides with the locus of points that are singular in their fiber. In particular, $\pi$ is smooth near every point of $\mc{C}$ which is not a node. Under the identification of  \eqref{eq:b-txyg}, the singular locus of $\pi$ near the node $q$ is defined by
\[
G=\frac{\partial G}{\partial x}=\frac{\partial G}{\partial y}=0,
\]
that is,
\[
xy=tu,\qquad y=t\frac{\partial u}{\partial x},\qquad x=t\frac{\partial u}{\partial y}.
\]
This implies
\[
t
\left(
u
-
t
\frac{\partial u}{\partial x}
\frac{\partial u}{\partial y}
\right)
=0
\]
on the singular locus of $\pi$. Since $u(0,0,0)\neq 0$, the factor $u
-
t
\frac{\partial u}{\partial x}
\frac{\partial u}{\partial y}
$
is a unit, because its value at the closed point is $u(0,0,0)\neq 0$. Therefore $t=0$ on the singular locus of $\pi$. Since the completion of a Noetherian local
ring is faithfully flat, $t$ is already zero in the local ring of the
relative singular locus at $q$. Therefore there is an open neighborhood
of $q$ in $\widetilde{\mc{C}}$ such that every point of the relative singular
locus in that neighborhood lies over $s\in Z$. This implies the conclusion.
\end{proof}

\subsection{An arithmetic variant of the Graber--Harris--Starr comb smoothing theorem}

The following theorem is an arithmetic version of \cite[\S2.1, Lemma~2.6, and paragraph below Hypothesis~2.7]{graber2003families}.

\begin{thm}\label{thm:smoothing} Let $k$ be an infinite field.
Let $X$ be a smooth projective $k$-variety, let $D$ be a smooth projective curve over $k$, and let $Z \subset D$ be a finite closed subscheme. Let $\mathcal{C}$ be a comb with handle $D$ and teeth $(C_i)_{1 \leqslant i \leqslant m}$ over $k$, and let
$\iota \colon \mathcal{C} \hookrightarrow X$ be a closed immersion such that the
restrictions $C_i \to X$ are free rational curves whose images do not meet $\iota(Z)$.
Suppose that $H^1(D,\, \mathcal{N}_{\mathcal{C}/X}|_{D} \otimes \mathcal{O}_D(-Z)) = 0$ and that $\mathcal{N}_{\mathcal{C}/X}|_D \otimes \mathcal{O}_D(-Z)$ is globally generated.
Then there exists a smoothing of $\iota$ fixing $\iota(Z)$, in the sense of \Cref{def:smoothing}.
\end{thm}

\begin{proof}
If $m=0$, the constant family gives the required smoothing. From now on, we assume $m\geq 1$. Let us consider the closed subscheme $\mathrm{Hilb}_{X,Z}\subset \mathrm{Hilb}_{X}$ parameterizing closed subschemes of $X$ containing $Z$: for every $k$-scheme $S$, we have
\[
\mathrm{Hilb}_{X,Z}(S)=\{\,W\subset X\times_k S \mid W \text{ is a closed subscheme flat over } S,\; Z\times_k S\subset W\,\}.
\]

Since $H^1(D,\, \mathcal{N}_{\mathcal{C}/X}|_{D} \otimes \mathcal{O}_D(-Z)) = 0$,
the subscheme $Z$ does not meet the $(C_i)_{1 \leqslant i \leqslant m}$, and the $C_i$ are
free curves on $X$, by \cite[II.7, Lemma~7.5]{kollar1996rational} we have
$H^1(\mathcal{C}, \mathcal{N}_{\mathcal{C}/X} \otimes \mathcal{O}_{\mathcal{C}}(-Z)) = 0$.
Thus the point $[\iota] \in \mathrm{Hilb}_{X,Z}(k)$ is smooth.

Let us denote by $q_i$ the intersection point of $C_i$ and $D$. Let us prove that
there exists an element of the tangent space
$H^0(\mathcal{C}, \mathcal{N}_{\mathcal{C}/X} \otimes \mathcal{O}_{\mathcal{C}}(-Z))$
whose image in $T_{C_i, q_i}$ is non-zero for every $i$. We have the exact sequence
\[
0 \to \mathcal{N}_{\mathcal{C}/X} \otimes \mathcal{O}_{\coprod C_i}(-q_1 - \cdots - q_m)
\to \mathcal{N}_{\mathcal{C}/X} \otimes \mathcal{O}_{\mathcal{C}}(-Z)
\to \mathcal{N}_{\mathcal{C}/X} \otimes \mathcal{O}_{\mathcal{C}}(-Z) \otimes \mathcal{O}_D
\to 0.
\]
Since the $C_i \to X$ are free rational curves, we have
$H^1(\coprod C_i, \mathcal{N}_{\mathcal{C}/X} \otimes \mathcal{O}_{\coprod C_i}(-q_1 - \cdots - q_m)) = 0$,
and thus the restriction map
\[
H^0(\mathcal{C}, \mathcal{N}_{\mathcal{C}/X} \otimes \mathcal{O}_{\mathcal{C}}(-Z))
\to H^0(D, \mathcal{N}_{\mathcal{C}/X}|_{D} \otimes \mathcal{O}_D(-Z))
\]
is surjective, and
\[
H^1(\mathcal{C}, \mathcal{N}_{\mathcal{C}/X} \otimes \mathcal{O}_{\mathcal{C}}(-Z))
\to H^1(D, \mathcal{N}_{\mathcal{C}/X}|_{D} \otimes \mathcal{O}_D(-Z))
\]
is bijective.

For every $i=1,\dots,m$, we have the exact sequence
\[
0 \to \mathcal{N}_{\mathcal{C}/X}|_{D} \otimes \mathcal{O}_D(-Z) \otimes \mathcal{O}_D(-q_i)
\to \mathcal{N}_{\mathcal{C}/X}|_{D} \otimes \mathcal{O}_D(-Z)
\to \mathcal{N}_{\mathcal{C}/X,q_i}
\to 0.
\]
Since $\mathcal{N}_{\mathcal{C}/X}|_D \otimes \mathcal{O}_D(-Z)$ is globally generated and $q_i \notin Z$, the homomorphism
\[H^0(D, \mathcal{N}_{\mathcal{C}/X}|_{D} \otimes \mathcal{O}_D(-Z))\otimes_k k(q_i)
\to \mathcal{N}_{\mathcal{C}/X,q_i}\]
is surjective. Since the map $\mathcal{N}_{\mathcal{C}/X,q_i} \to T_{C_i,q_i}$ is
also surjective, the composite
\[
H^0(\mathcal{C}, \mathcal{N}_{\mathcal{C}/X} \otimes \mathcal{O}_{\mathcal{C}}(-Z))\otimes_k k(q_i)
\to H^0(D, \mathcal{N}_{\mathcal{C}/X}|_{D} \otimes \mathcal{O}_D(-Z))\otimes_k k(q_i)
\to T_{C_i, q_i}
\]
is surjective for every $i$. Thus, since $k$ is infinite, there exists an element
$u \in H^0(\mathcal{C}, \mathcal{N}_{\mathcal{C}/X} \otimes \mathcal{O}_{\mathcal{C}}(-Z))$
whose image is non-zero in $T_{C_i, q_i}$ for every $i$. In particular, since $m\geq 1$, we have $u\neq 0$.

Let $T \subset \mathrm{Hilb}_{X,Z}$ be a smooth curve passing through the smooth $k$-point $[\iota]$ with
tangent direction $u$. This corresponds to a closed subscheme $Y \subset X \times_k T$
flat over $T$ containing $Z \times_k T$. The composite map
$T \to \mathrm{Hilb}_{X,Z} \hookrightarrow \mathrm{Hilb}_X$ corresponds to a curve
passing through $[\iota]$ with tangent direction given by the image of $u$ in
$H^0(\mathcal{C}, \mathcal{N}_{\mathcal{C}/X})$, and thus with a tangent direction
whose image in $T_{C_i, q_i}$ is non-zero for every $i$. By \Cref{prop:transverse-curve}, replacing $T$ by an open neighborhood of  $[\iota]$ in $T$ if necessary, we may assume that the map $Y \to T$ is smooth over $T \setminus \{[\iota]\}$. This gives the desired smoothing of $\iota$.
\end{proof}

\subsection{Fibrations over a curve and proof of Theorem~\ref{thm:add_enough_teeth}}

We now apply the comb smoothing technique to combs $C\to X$ whose handle $D$ is a section of a  projective morphism $X\to D$ which is smooth with separably rationally connected fibers over a nonempty open subscheme of $D$. 

The following lemma is an arithmetic analogue of a result from \cite[Section 2.2]{graber2003families}; see \cite[Lemma 10 and Theorem 16]{Kollar2004Specialisation}.

\begin{lemma}\label{existence_comb_smoothable_intermediaire}
Let $k$ be an infinite field, let $X$ be a smooth projective $k$-scheme, let $D$ be a smooth projective irreducible $k$-curve, and let $p\colon X\to D$ be a projective morphism, and let $D^\circ\subset D$ be a nonempty open subscheme such that $p^{-1}(D^\circ)\to D^\circ$ is smooth with separably rationally connected fibers. Set
\[
W:=X\times_k \mathbb P^3_k\times_k \mathbb P^1_k,
\]
and denote by $\mathrm{pr}_1\colon W\to X$, $\mathrm{pr}_2\colon W\to \mathbb{P}^3_k$ and $\mathrm{pr}_3\colon W\to \mathbb P^1_k$ the first, second and third projections, respectively. Let $Z\subset D$ be a finite closed subscheme, let $\mathcal C$ be a comb with handle $D$ and teeth $(C_i)_{1\le i\le m}$, all disjoint from $Z$, and let $M$ be a line bundle on $D$. Let $\iota\colon \mathcal C\hookrightarrow W$ be a closed immersion such that:
\begin{enumerate}[label=\textup{(H\arabic*)}]
\item $\mathrm{pr}_1\circ\iota|_D\colon D\to X$ is a section of $p$, and $\mathrm{pr}_3\circ\iota|_D\colon D\to\mathbb P^1_k$ is the constant map with value $0$;

\item for every $i$, the tooth $C_i$ is free in $W$, and $\mathrm{pr}_3\circ\iota|_{C_i}\colon C_i\to\mathbb P^1_k$ is étale;

\item for every $i$, the composite $C_i\xrightarrow{\iota|_{C_i}}W\xrightarrow{\mathrm{pr}_1}X\xrightarrow{p}D$ is constant; denoting by $d_i\in D$ its image and setting $X_{d_i}\coloneqq X\times_D\Spec k(d_i)$, we have that $d_i\in D^\circ$, and that the $k(d_i)$-rational curve
\[
C_i\xrightarrow{\iota|_{C_i}}
X_{d_i}\times_k\mathbb P^3_k\times_k\mathbb P^1_k
\xrightarrow{\mathrm{pr}_1}
X_{d_i}
\]
is very free.
\end{enumerate}

Then there exists a comb $\mathcal C'=\mathcal C\cup C_{m+1}$ with handle $D$ together with a closed immersion $\iota'\colon\mathcal C'\hookrightarrow W$ extending $\iota$, such that:
\begin{enumerate}
\item $\mathrm{pr}_3\circ\iota'|_{C_{m+1}}\colon C_{m+1}\to\mathbb P^1_k$ is étale;

\item $C_{m+1}\cap Z=\varnothing$;

\item the composite $C_{m+1}\xrightarrow{\iota'_{|C_{m+1}}} W \xrightarrow{\mathrm{pr}_1}X\xrightarrow{p}D$ is constant, with image denoted by $d_{m+1}$; we have $d_{m+1}\in D^\circ$; moreover, the tooth $C_{m+1}$ is free in $W$, and the induced morphism $C_{m+1}\to X_{d_{m+1}}$ is very free;

\item if $H^1(D,\mathcal N_{\mathcal C/W}|_D \otimes  M)\neq 0$, then
\[
\dim_k H^1\bigl(D,\mathcal N_{\mathcal C'/W}|_D\otimes M\bigr)
<
\dim_k H^1\bigl(D,\mathcal N_{\mathcal C/W}|_D\otimes M\bigr).
\]
\end{enumerate}
\end{lemma}

\begin{proof}
The proof follows the argument of \cite[Lemma 10]{Kollar2004Specialisation}. By Serre duality,
\[
H^1(D,\, \mathcal N_{\mathcal C/W}|_D \otimes M)^\vee
\simeq
\operatorname{Hom}(M \otimes \omega_D^{-1}, (I_{\mathcal C}/I_{\mathcal C}^2)|_D).
\]
Let $\varphi \in \operatorname{Hom}(M \otimes \omega_D^{-1}, (I_{\mathcal C}/I_{\mathcal C}^2)|_D)\setminus \{0\}$, and let $U\subset D$ be the dense open subscheme over which $\varphi$ has rank one. Replacing $U$ by $U\cap D^\circ$, we may assume that $U\subset D^\circ$. We apply \Cref{avf_passing_Z} to the smooth projective morphism $X_{D^\circ}\times_k\mathbb P^3_k\to D^\circ$ and to the image of $(\mathrm{pr}_{1,2}\circ\iota)(D^\circ)$, where $\mathrm{pr}_{1,2}\colon W\to X\times_k \mathbb P^3_k$ is the projection. Identifying $D^\circ$ with its image $(\mathrm{pr}_{1,2}\circ\iota)(D^\circ)$, there exists a closed point $z \in U$, with $k(z)/k$ separable, avoiding $Z$ 
and the points $C_i \cap D$, together with a basis
$u_1, \dots, u_n \in (\mathcal{N}_{D/X \times_k \mathbb{P}^3_k})_z$
such that for each $j=1,\dots,n$ there exists a free curve $f_j \colon \mathbb{P}^1_{k(z)} \to X \times_k \mathbb{P}^3_k$
satisfying $f_j(0) = z$, such that $p \circ \on{pr}_1 \circ f_j$ is constant, the vector 
$T_0(f_j)(\partial/\partial t)$ maps to $u_j$ via the map 
$(\mathcal{T}_{X \times_k \mathbb{P}^3_k / D})_z \to (\mathcal{N}_{D / X \times_k 
\mathbb{P}^3_k})_z$, and such that $(f_j)_{p(z)} \colon \mathbb{P}^1_{k(z)} \to 
(X \times_k \mathbb{P}^3_k)_{p(z)}$ is a very free closed immersion.

Since $k$ is infinite, for every $j=1,\dots,n$ we can choose a closed immersion
$F_j\colon \mathbb P^1_{k(z)}\to W$
such that $F_j(0)=z$, the image of $F_j$ meets $\mathcal C$ only at $z$, its projection to $X\times_k\mathbb P^3_k$ equals $f_j$, and $\mathrm{pr}_3\circ F_j$ is étale.
In addition, since $k$ is infinite, we may choose another closed immersion
$F_0\colon \mathbb P^1_{k(z)}\to W$
satisfying the same properties as $F_1$, but such that
\[
T_0(\mathrm{pr}_3\circ F_0)(\partial/\partial t)\neq T_0(\mathrm{pr}_3\circ F_1)(\partial/\partial t).
\]
It follows that the images of
\[
(T_0(F_0)(\partial/\partial t),T_0(F_1)(\partial/\partial t),\dots,T_0(F_n)(\partial/\partial t))
\]
in $N_{z,D}\coloneqq (\mathcal N_{D/W})_z$ form a basis of $N_{z,D}$, and each $F_j$ induces a very free curve in the fiber $W_{p(z)}$.

Let $\mathcal C(F_j)$ be the comb obtained by attaching the tooth $F_j$ to $\mathcal C$, and let $L_j(z)$ be the tangent line of $F_j$ at $z$.
Let $r_z\colon (I_{\mathcal C}/I_{\mathcal C}^2)|_D\to N_{z,D}^*$ be restriction, and let $q_{j,z}\colon N_{z,D}^*\to L_j(z)^*$ be dual to the inclusion $L_j(z)\subset N_{z,D}$. We obtain an exact sequence
\[
0\to (I_{\mathcal C(F_j)}/I_{\mathcal C(F_j)}^2)|_D
\to (I_{\mathcal C}/I_{\mathcal C}^2)|_D
\xrightarrow{q_{j,z}\circ r_z}
L_j(z)^*
\to 0.
\]
The map
\[
r_z\circ \varphi\colon M\otimes \omega_D^{-1}\otimes_{\mathcal O_D} k(z)\to N_{z,D}^*
\]
is injective by the choice of $z$.
Since $(T_0(F_0)(\partial/\partial t),\dots,T_0(F_n)(\partial/\partial t))$ is a basis of $N_{z,D}$, there exists $j$ such that $q_{j,z}\circ r_z\circ \varphi$ is injective. Choose such $j$, define $C_{m+1}$ as the corresponding tooth, and set $\mathcal C'=\mathcal C(F_j)$.
We obtain
\begin{equation}\label{reduction_coho_eq}
\dim_k H^1(D,\mathcal N_{\mathcal C'/W}|_D\otimes M)
<
\dim_k H^1(D,\mathcal N_{\mathcal C/W}|_D\otimes M).
\end{equation}
Since $(f_j)_{p(z)}$ is very free, the composite
\[
\mathbb{P}^1_{k(z)} \xrightarrow{(f_j)_{p(z)}} X_{p(z)} \times_k \mathbb{P}^3_k 
\xrightarrow{\mathrm{pr}_1'} X_{p(z)}
\]
is very free. Since $\mathrm{pr}_1 \circ F_j = \mathrm{pr}_1' \circ f_j$, where 
$\mathrm{pr}_1' \colon X \times_k \mathbb{P}^3_k \to X$ denotes the first projection,
the induced morphism $C_{m+1} \to X_{p(z)}$ is very free, verifying condition~(3).
Conditions~(1) and~(2) follow from the construction of $F_j$, and condition~(4) 
is \eqref{reduction_coho_eq}.
\end{proof}

The next lemma will be used to construct combs satisfying the hypotheses of \Cref{thm:smoothing}.

\begin{lemma}\label{existence_comb_smoothable}
Let $k$ be an infinite field, let $X$ be a smooth projective $k$-scheme, let $D$ be a smooth projective geometrically connected $k$-curve, and let
$p \colon X \to D$ be a projective morphism. Let $D^\circ\subset D$ be a nonempty open subscheme such that $p^{-1}(D^\circ)\to D^\circ$ is smooth with separably rationally connected fibers. Set $W \coloneqq X \times_k \mathbb{P}^3_k \times_k \mathbb{P}^1_k$, and denote by $\on{pr}_1 \colon W \to X$, $\on{pr}_2 \colon W \to \mathbb{P}^3_k$, and $\on{pr}_3 \colon W \to \mathbb{P}^1_k$ the natural projections. Let $Z \subset D$ be a finite closed subscheme, and let
$\iota \colon D \hookrightarrow W$
be a closed immersion such that:
\begin{enumerate}[label=\textup{(H\arabic*)}]
    \item the composition
    $D \xhookrightarrow{\iota} W \xrightarrow{\on{pr}_3} \mathbb{P}^1_k$
    is the constant map with value $0 \in \mathbb{P}^1_k$;

    \item the composition
    $D \xhookrightarrow{\iota} W \xrightarrow{\on{pr}_1} X$
    is a section of $p$.
\end{enumerate}

Then, for every integer $N\geqslant 0$, there exist a comb $\mathcal{C}$ with handle $D$ and teeth $(C_i)_{1\leq i\leq m}$, with $m>N$, all attached at points of $D\setminus Z$, and a closed immersion
$\iota_{\mathcal{C}}\colon \mathcal{C}\hookrightarrow W$ extending $\iota$, such that:
\begin{enumerate}
    \item for every $i$, the morphism
    $\on{pr}_3 \circ \iota_{\mathcal{C}}|_{C_i}\colon C_i\to\mathbb{P}^1_k$
    is étale;

    \item
    \[
    H^1\!\left(
    D,\,
    \mathcal{N}_{\mathcal{C}/W}\big|_{D}
    \otimes \mathcal{O}_D(-Z)
    \right)=0;
    \]

    \item $\mathcal{N}_{\mathcal{C}/W}|_D \otimes \mathcal{O}_D(-Z)$ is globally generated;

    \item for every $i$, the composite
    \[
    C_i \hookrightarrow \mathcal{C}
    \xrightarrow{\iota_{\mathcal{C}}} W
    \xrightarrow{\on{pr}_1} X
    \xrightarrow{p} D
    \]
    is constant. If $d_i\in D$ denotes its image, then $d_i\in D^\circ$, and the induced morphism
    \[
    C_i \to X_{d_i}\coloneqq X\times_D\operatorname{Spec} k(d_i)
    \]
    is very free.
\end{enumerate}
\end{lemma}

\begin{proof}
Applying \Cref{existence_comb_smoothable_intermediaire} repeatedly with the sheaf $M=\mathcal{O}(-Z)$, we obtain a comb $\mathcal{C}_1$ with handle $D$ and teeth not meeting $Z$, together with a closed immersion
$\iota_{\mathcal{C}_1}\colon \mathcal{C}_1\hookrightarrow X \times_k \mathbb{P}^3_k \times_k \mathbb{P}^1_k$
extending $\iota$, satisfying conditions (1), (2), and (4) of the statement, and such that $\iota_{\mathcal{C}_1}$ satisfies (H1), (H2), and (H3) of \Cref{existence_comb_smoothable_intermediaire}.
Fix a very ample line bundle $\mathcal{L}$ on $D$. Applying \Cref{existence_comb_smoothable_intermediaire} repeatedly again, now with $M=\mathcal{O}_D(-Z) \otimes\mathcal{L}^{-1}$, we can add teeth to $\mathcal{C}_1$ and obtain a comb $\mathcal{C}$ with handle $D$, teeth not meeting $Z$, and more than $N$ teeth, together with a closed immersion
$\iota_{\mathcal{C}}\colon \mathcal{C}\hookrightarrow X \times_k \mathbb{P}^3_k \times_k \mathbb{P}^1_k$
extending $\iota_{\mathcal{C}_1}$, such that $\mathcal{C}$ satisfies conditions (1) and (4) of the statement, and such that
\[
H^1\!\left(
D,\,
\mathcal{N}_{\mathcal{C}/W}\big|_D\otimes\mathcal{O}_D(-Z) \otimes\mathcal{L}^{-1}
\right)=0.
\]
By the Castelnuovo--Mumford criterion (see \cite[08A8]{stacks-project}), it follows that $\mathcal{N}_{\mathcal{C}/W}|_D \otimes \mathcal{O}_D(-Z)$ is globally generated. Hence $\mathcal{C}$ satisfies (3). Moreover, adding teeth can only decrease the dimension of the relevant $H^1$, so $\mathcal{C}$ still satisfies (2). Thus $\mathcal{C}$ is the desired comb.
\end{proof}

\begin{rmk}
The arguments used in the proofs of Lemmas~\ref{existence_comb_smoothable_intermediaire} and~\ref{existence_comb_smoothable} are similar to those appearing in the proof of \cite[Theorem~27]{AraujoKollar2003}.
\end{rmk}

We are now ready to prove \Cref{thm:add_enough_teeth}.

\begin{proof}[Proof of \Cref{thm:add_enough_teeth}]
As $D$ is a smooth curve, it admits a closed immersion $j \colon D \hookrightarrow 
\mathbb{P}^3_k$; see \cite[Chapter IV, Corollary 3.6]{Hartshorne1977algebraic}. Consider the closed immersion
\[
\iota = (s,j,0) \colon D \hookrightarrow X \times_k \mathbb{P}^3_k \times_k \mathbb{P}^1_k.
\]
By \Cref{existence_comb_smoothable}, with the open subscheme $D^\circ\subset D$, there exists a comb $\mathcal{C}$ with handle $D$ 
and teeth $(C_i)_{1 \leqslant i \leqslant m}$, with $m > N$, all attached at points of 
$D \setminus Z$, together with a closed immersion 
\[
\iota_{\mathcal{C}} \colon \mathcal{C} \hookrightarrow X \times_k \mathbb{P}^3_k 
\times_k \mathbb{P}^1_k,
\]
such that the following conditions hold. For each $i$, the composite
\[
C_i \xhookrightarrow{\iota_{\mathcal{C}}|_{C_i}} X \times_k \mathbb{P}^3_k \times_k 
\mathbb{P}^1_k \to X \to D
\]
is constant with image some point $d_i \in D^\circ$, and the $k(d_i)$-rational curve obtained by the composite
\[
C_i \to X_{d_i} \times_k \mathbb{P}^3_k \times_k \mathbb{P}^1_k \to X_{d_i}
\]
is very free, where $X_{d_i} \coloneqq X \times_D \operatorname{Spec} k(d_i)$. By \Cref{vf_avf}, the composite $C_i\to X$ is almost very free for every $i$. Moreover, $H^1(D,\, \mathcal{N}_{\mathcal{C}/X \times_k \mathbb{P}^3_k \times_k 
\mathbb{P}^1_k}|_D \otimes \mathcal{O}_D(-Z)) = 0$ and $\mathcal{N}_{\mathcal{C}/X \times_k \mathbb{P}^3_k \times_k 
\mathbb{P}^1_k}|_D \otimes \mathcal{O}_D(-Z)$ is globally generated.
By \Cref{thm:smoothing}, the closed immersion $\iota_{\mathcal{C}} \colon \mathcal{C} \to 
X \times_k \mathbb{P}^3_k \times_k \mathbb{P}^1_k$ is smoothable fixing $\iota_{\mathcal{C}}(Z)$, and the same smoothing yields a 
smoothing of the composite $\mathcal{C} \to X \times_k \mathbb{P}^3_k \times_k 
\mathbb{P}^1_k \to X$ fixing $s(Z)$. The comb $\mathcal{C}$ therefore satisfies all the 
conditions of the statement.
\end{proof}

\section{Weil restrictions and Henselian lifting}\label{sec:5}

\subsection{Smooth points of Weil restrictions}

Let $S$ be a scheme, let $\mc{X}$ and $\mc{Y}$ be flat projective $S$-algebraic spaces locally of finite presentation, and let $p\colon \mc{X} \to \mc{Y}$ be a projective morphism. Recall that the Weil restriction functor $\mathrm{Res}_{\mc{Y}/S}(\mc{X})$ is the presheaf which associates to a $S$-scheme $S'$ the set
\[\mathrm{Res}_{\mc{Y}/S}(\mc{X})(S')=\on{Hom}_\mc{Y}(\mc{Y}\times_S S',\mc{X}),\]
where $\mc{Y}\times_S S'$ is viewed as a $\mc{Y}$-scheme via the first projection $\mc{Y}\times_S S'\to \mc{Y}$. The functor $\mathrm{Res}_{\mc{Y}/S}(\mc{X})$ is representable by a $S$-algebraic space locally of finite type: by Noetherian approximation, one reduces to the case when $S$ is Noetherian, in which case the statement is proved in \cite[Theorem 1.5]{Olsson2006underline}. For a $S$-scheme $S'$ and a $\mc{Y}$-morphism $f\colon \mc{Y}\times_S S'\to \mc{X}$, we write $[f]$ for the corresponding $S'$-point of $\mathrm{Res}_{\mc{Y}/S}(\mc{X})$.

\begin{lemma}\label{smoothness_weil}
Let $S$ be a scheme, let $\mc{Y}$ be  a smooth projective $S$-scheme, and let $p\colon \mc{X} \to \mc{Y}$ be a proper finitely presented morphism of $S$-algebraic spaces. Let $k$ be a field equipped with a map $\operatorname{Spec} k \to S$, and let $f \colon \mc{Y}_k \to \mc{X}_k$ be a section of $p_k \colon \mc{X}_k \to \mc{Y}_k$. Assume that $p$ is smooth in an open neighborhood of $f(\mc{Y}_k)$. If $H^1(\mc{Y}_k, f^*\mathcal{T}_{\mc{X}/\mc{Y}}) = 0$, then the image of $[f]\in \mathrm{Res}_{\mc{Y}/S}(\mc{X})(k)$ lies in the smooth locus of the algebraic space $\mathrm{Res}_{\mc{Y}/S}(\mc{X})$.
\end{lemma}

\begin{proof}
This follows from the standard obstruction theory for the Hom-stack.
Indeed, the obstruction to lifting $f$ across a square-zero thickening
with ideal $I$ lies in $\operatorname{Ext}^1_{\mc{Y}_k}
        \bigl(f^*L_{\mc{X}/\mc{Y}},\mc{O}_{\mc{Y}_k}\otimes_k I\bigr)$; see \cite{aoki2006hom, aoki2006hom-erratum} or \cite[Theorem 1.5]{olsson2006deformation}. Since $p$ is smooth in an open neighborhood of $f(\mc{Y}_k)$, the pullback $f^*L_{\mc{X}/\mc{Y}}$ is represented by $f^*\Omega_{\mc{X}/\mc{Y}}$ in degree zero; see \cite[Lemmas 08R5 and 08VE]{stacks-project}.
        Moreover, $f^*\Omega_{\mc{X}/\mc{Y}}$ is finite locally free; see \cite[Lemma 0CK5]{stacks-project}. Hence
\[
\begin{aligned}
        \operatorname{Ext}^1_{\mc{Y}_k}
        (f^*L_{\mc{X}/\mc{Y}},\mc{O}_{\mc{Y}_k}\otimes_k I)
        &\cong\operatorname{Ext}^1_{\mc{Y}_k}(f^*\Omega_{\mc{X}/\mc{Y}},\mc{O}_{\mc{Y}_k}\otimes_k I)  \\
        &\cong H^1(\mc{Y}_k,\mathcal Hom(f^*\Omega_{\mc{X}/\mc{Y}},\mc{O}_{\mc{Y}_k}\otimes_k I))\\
        &\cong H^1(\mc{Y}_k,f^*\mathcal T_{\mc{X}/\mc{Y}}\otimes_k I)\\
        &\cong H^1(\mc{Y}_k,f^*\mathcal T_{\mc{X}/\mc{Y}})\otimes_k I\\ 
        &\cong 0,
\end{aligned}
\]
where the second isomorphism uses that $f^*\Omega_{\mc{X}/\mc{Y}}$ is locally free, and the last isomorphism uses the assumption. Therefore
$\operatorname{Res}_{\mc{Y}/S}(\mc{X})$ is formally smooth at $[f]$. Since the
Weil restriction is locally of finite presentation over $S$, formal
smoothness at $[f]$ implies smoothness at $[f]$; see
\cite[Lemma 04AM]{stacks-project}. Hence $[f]$ lies in the smooth locus
of $\operatorname{Res}_{\mc{Y}/S}(\mc{X})$.
\end{proof}

\begin{prop}\label{prop_principal_Weil}
Let $S$ be a scheme, let $C\to S$ be a smooth projective morphism with geometrically connected fibers of dimension $1$, let $X$ be a smooth proper $S$-algebraic space, and let $p\colon X\to C$ be an $S$-morphism. Let $k$ be a field endowed with a morphism $\operatorname{Spec} k \to S$, such that $X_k$ is a scheme and such that the morphism $p_k  \colon X_k \to C_k$ is projective and there exists a nonempty open subscheme $C_k^\circ\subset C_k$ such that $p_k^{-1}(C_k^\circ)\to C_k^\circ$ is smooth with separably rationally connected fibers. Let $Z \subset C_k$ be a finite subscheme of $C_k$. Assume that there exists a $k$-point of $\mathrm{Res}_{C/S}(X)$ and denote by $s \colon C_k \to X_k$ the corresponding section.

\begin{enumerate}
    \item There exists a finite family $\{k_i/k\}_{i\in I}$ of finite separable extensions of collectively coprime degrees such that for each $i\in I$ there exists a smooth $k_i$-point of $\mathrm{Res}_{C/S}(X)$ whose associated section $s_i \colon C_{k_i} \to X_{k_i}$ coincides with $s_{k_i}$ on $Z_{k_i}$.
    \item Assume further that the field $k$ is large. Then there exists a smooth $k$-point of $\mathrm{Res}_{C/S}(X)$ whose associated section $s' \colon C_{k}\to X_{k}$ coincides with  $s$ on $Z$.
\end{enumerate}
\end{prop}

\begin{proof} 
Suppose first that $k$ is infinite. Let $X^{\mathrm{sm}}\subset X$ be the smooth locus of $p$, and let $t\in S$ be the image of $\operatorname{Spec} k\to S$. By \Cref{smooth_along_section}, the morphism $p_k$ is smooth along $s(C_k)$. We have $X_k=X_t\times_{\kappa(t)}k$. Since $k/\kappa(t)$ is faithfully flat and the formation of the smooth locus commutes with flat base change, the morphism $p_t\colon X_t\to C_t$ is smooth at the images in $X_t$ of the points of $s(C_k)$; see \cite[Lemma~0DZI]{stacks-project}. In particular, $p_t$ is flat at these points. Since both $X\to S$ and $C\to S$ are smooth, they are flat, and the fiberwise criterion for flatness applied to $X\to C\to S$ shows that $p$ is flat along the image of $s(C_k)$ in $X$; see \cite[Theorem~05X0]{stacks-project}. Finally, $p$ is locally of finite presentation, so \cite[Lemmas~01V9 and 03ZF]{stacks-project} show that $p$ is smooth along $s(C_k)$. Moreover, $p$ is smooth at every point of $p_k^{-1}(C_k^\circ)$. Thus $s(C_k)\cup p_k^{-1}(C_k^\circ)\subset X_k^{\mathrm{sm}}$. In particular, $s^*\mathcal{T}_{X^{\mathrm{sm}}/C}=s^*\mathcal{T}_{X/C}$ is a vector bundle on $C_k$. 

By \Cref{Serre_ampleness}, there exists a positive integer $N$ such that $H^1(C_k, M \otimes s^*\mathcal{T}_{X/C})=0$ for every line bundle $M$ with $\mathrm{deg} M \geqslant N$. By \Cref{thm:add_enough_teeth}, applied with the open subscheme $C_k^\circ\subset C_k$, there exists a comb $\mathcal{C}$ with handle $D \coloneqq  C_k$ and $m' \geqslant N$ teeth $\{C'_i\}_{1 \leqslant i \leqslant m'}$, together with a map $f \colon \mathcal{C} \to X_k$ with $f|_{C_k}=s$ and $f|_{C'_i}$ almost very free with $p \circ (f|_{C'_i})$ constant and a smoothing of $f$ fixing $f(Z)$ as in \Cref{def:smoothing}. The attachment points of the teeth belong to $C_k^\circ$, hence $f(\mathcal C)\subset X_k^{\mathrm{sm}}$. Since $Y\to T$ is projective, after shrinking $T$ around $t_0$ we may assume that $F(Y)\subset X_k^{\mathrm{sm}}$. Moreover, $(f|_{C'_i})^*\mathcal{T}_{X/C}=(f|_{C'_i})^*\mathcal{T}_{X^{\mathrm{sm}}/C}$ is ample by \Cref{vf_avf}.
By the choice of $m'$ and \cite[Lemma 7.10.1]{kollar1996rational}, we have $H^1(Y_t, F_t^*\mathcal{T}_{X/C}) = 0$ for a general point $t \in T$. By \Cref{continuity_smoothing}, shrinking $T$ again, we may assume that for every $t \in T \setminus \{t_0\}$, the fiber $Y_t$ is isomorphic to $C_{k(t)}$ and $F_t$ defines a section $C_{k(t)} \to X_{k(t)}$ of $p_{k(t)} \colon X_{k(t)} \to C_{k(t)}$. Moreover, since the smoothing fixes $Z$, we have ${F_t}|_{Z_{k(t)}} = {s_{k(t)}}|_{Z_{k(t)}}$.

Since $T(k) \neq \emptyset$, the smooth curve $T \setminus \{t_0\}$ contains a zero-cycle of degree $1$, and hence there exists a finite set $I$ and closed points $\{t_i\}_{i\in I}$ whose residue fields $k_i \coloneqq k(t_i)$ have degrees over $k$ that are coprime. By \cite[Theorem~9.2]{gabber2013index}, we may choose the points $t_i$ such that each extension $k_i/k$ is separable. If $k$ is large, we may choose $I=\{i\}$ to be a singleton and let $t_i\in T \setminus \{t_0\}$ be a $k$-point. By \Cref{continuity_smoothing}, each $Y_{t_i}$ is isomorphic to $C_{k_i}$, and $F_{t_i}$ defines a section $C_{k_i} \to X_{k_i}$ which coincides with $s_{k_i}$ when restricted to $Z_{k_i}$. The vanishing $H^1(C_{k_i}, F_{t_i}^*\mathcal{T}_{X/C}) = 0$ implies, by \Cref{smoothness_weil}, that the corresponding $k_i$-point of $\mathrm{Res}_{C/S}(X)$ is smooth, which concludes the proof of \Cref{prop_principal_Weil} in the case when $k$ is infinite.

It remains to prove (1) when $k$ is finite. By an observation of Colliot-Thélène \cite[Theorem 1.3]{pop2014little}, for every prime $p$, the $p$-closure $k^{(p)}$ of $k$ is a large field. Now (2) applied to $k^{(p)}$ implies the existence of a finite extension $l_p/k$ of degree prime to $p$ and of a smooth $l_p$-point of $\mathrm{Res}_{C/S}(X)$. Since $k$ is finite, it is perfect, and hence the finite extension $l_p/k$ is separable. Thus (1) holds for $k$, as desired.
\end{proof}

\subsection{Proof of Theorem \ref{intro:cor_prop_principal_Weil}}

We will need the following lemma.

\begin{lemma}\label{henselian_lemma}
Let $(R,\mathfrak{m})$ be a Henselian local ring with residue field $\kappa$. Then for every finite extension $\kappa \to \kappa'$, there exists a Henselian local ring $(R',\mathfrak{m}')$ and a finite local homomorphism $f_R \colon R \to R'$ such that $\mathfrak{m}R' = \mathfrak{m}'$ and the induced map on residue fields
$f_R \otimes_R \kappa \colon  \kappa \to \kappa'$
is the given extension. Moreover, if $\kappa \to \kappa'$ is separable, we can take $R \to R'$ to be finite \'etale.
\end{lemma}

\begin{proof}
As a finite local ring over a Henselian ring, $R'$ is again Henselian. 
Moreover, since $\kappa \to \kappa'$ is the composition of a separable extension and a purely inseparable extension, 
it suffices to treat these two cases separately.
If $\kappa \to \kappa'$ is separable, the statement follows directly from \cite[Lemma 04GK]{stacks-project}. 
If $\kappa \to \kappa'$ is a finite succession of extensions of the form $(\kappa[X]/(X^p - a))/\kappa$, 
we can reduce to the case where $\kappa' = \kappa[X]/(X^p - \overline{a})$ with $\overline{a} \in \kappa \setminus \kappa^p$. 
Let $a \in R$ be a lift of $\overline{a}$. Then 
$R' = R[X]/(X^p - a)$
satisfies the required condition.
\end{proof}

The following theorem implies \Cref{intro:cor_prop_principal_Weil}.

\begin{thm}\label{cor_prop_principal_Weil}
Let $R$ be a Henselian local ring with residue field $\kappa$. Let
$C \to \Spec R$ be a smooth projective morphism whose fibers are
geometrically connected of dimension $1$, and let $Z \subset C_{\kappa}$
be a subscheme finite over $\kappa$. Let $X$ be an $R$-algebraic space
such that $X_{\kappa}$ is a scheme, and assume that $X$ is smooth and proper over $R$. Let $p\colon X\to C$ be an $R$-morphism such that $p_\kappa$ is projective and admits a rational section. Denote by $s\colon C_\kappa\to X_\kappa$ its unique extension to a section. Assume moreover that there exists a nonempty open subscheme $C_\kappa^\circ\subset C_\kappa$ such that $p_\kappa^{-1}(C_\kappa^\circ)\to C_\kappa^\circ$ is smooth and has separably rationally connected fibers.
\begin{enumerate}
    \item There exist finite \'etale Henselian local $R$-algebras
$\{R_i\}_{i\in I}$ of collectively coprime degrees such that, for each $i\in I$, the base change $X_{R_i} \to C_{R_i}$ has a section $s_i$ such that $s_i|_{Z_{\kappa_i}} = s_{\kappa_i}|_{Z_{\kappa_i}}$, where $\kappa_i$ is the residue field of $R_i$.
    \item Assume further that the field $\kappa$ is large. Then there exists a section $s'$ of $p \colon X \to C$ such that ${s'}_{|Z}=s_{|Z}$.
\end{enumerate}
\end{thm}
\begin{proof}
The rational section extends uniquely to a section $s\colon C_\kappa\to X_\kappa$, since $C_\kappa$ is a smooth curve and $p_\kappa$ is proper.
(1) We apply \Cref{prop_principal_Weil}(1) to $s \colon C_{\kappa} \to X_\kappa$ and $Z \subset C_{\kappa}$.
This yields a finite collection $\{\kappa_i/\kappa\}_{i\in I}$ of finite extensions of collectively coprime degrees such that, 
for each $i$, there exists a $\kappa_i$-point $x_i \in \mathrm{Res}_{C/R}(X)(\kappa_i)$ which is smooth and whose associated section $s_i \colon  C_{\kappa_i} \to 
X_{\kappa_i}$ is such that ${s_i}_{|Z_{\kappa_i}}={s_{\Spec \kappa_i}}_{|Z_{\kappa_i}}$.
By \Cref{henselian_lemma}, we can lift each extension $\kappa_i/\kappa$ to a local finite \'etale extension 
$R \to R_i$ whose residue field is $\kappa_i$. In particular, $R_i$ is also a Henselian local ring.
As $x_i$ is a smooth $\kappa_i$-point of the algebraic space $\mathrm{Res}_{C/R}(X)$, by Hensel's lemma for algebraic spaces \cite[Theorem A.11]{Rydh2010Submersions} it lifts to an $R_i$-point. 
This defines a section $s_i \colon  C_{R_i} \to X_{R_i}$. The degrees of the finite \'etale $R$-algebras $R_i$ are collectively coprime, since they coincide with the degrees of the extensions $\kappa_i/\kappa$.

(2) The proof is identical to that of (1), using \Cref{prop_principal_Weil}(2) in place of \Cref{prop_principal_Weil}(1).
\end{proof}

\begin{cor}\label{intro:prop_projective_homogeneous_space}
Let $R$ be a Henselian local ring with residue field $\kappa$, let $C$ be a smooth projective $R$-curve with geometrically connected fibers, let $Y \to C$ be a smooth projective morphism such that, for every $c \in C$, the fiber $Y_c$ is a homogeneous space under a smooth affine connected $k(c)$-group, and such that there exists a rational section $C_\kappa \dashrightarrow Y_\kappa$.
\begin{enumerate}
    \item There exists a finite collection of finite \'etale Henselian local $R$-algebras $\{R_i\}_{i \in I}$ of collectively coprime degrees such that, for each $i$, the base change $Y_{R_i} \to C_{R_i}$ admits a section.
    \item If $\kappa$ is a large field, then $Y \to C$ admits a section.
\end{enumerate}
\end{cor}

\begin{proof}
First, since $Y_\kappa$ is projective, the rational section extends to a
section $C_\kappa \to Y_\kappa$.

Next, every fiber of $Y \to C$ is separably rationally connected. Fix
$c \in C$. By \Cref{src_invariant_base_change}, separable rational
connectedness of $Y \times_C \Spec k(c)$ may be checked after base change to an algebraic closure $\overline{k(c)}$ of $k(c)$. Now
$Y \times_C \Spec \overline{k(c)}$ is isomorphic to $G/P$, where $G$ is a
split reductive $\overline{k(c)}$-group and $P \subset G$ a parabolic subgroup. By the Bruhat
decomposition \cite[Exp.~XXVI, 4.3.2\,(iv')]{SGA3}, $G/P$ contains a dense
open subscheme isomorphic to an affine space; it is therefore separably
unirational, hence separably rationally connected by \Cref{unirational_src}.

\begin{enumerate}
    \item We apply \Cref{cor_prop_principal_Weil}(1) to $X = Y$ and $Z = \varnothing$. This yields finite \'etale extensions $\{R_i\}_{i \in I}$ of $R$ of collectively coprime degrees and sections $s_i \colon C_{R_i} \to Y_{R_i}$, as desired.
    \item The proof is identical to that of~(1), using \Cref{cor_prop_principal_Weil}(2) in place of \Cref{cor_prop_principal_Weil}(1). \qedhere
\end{enumerate}
\end{proof}

\begin{proof}[Proof of \Cref{cor:drinfeld-simpson}]
The morphism $(E/B)_\kappa \to C_\kappa$ admits a rational section, and since it is projective, it actually admits a section. Then, by \Cref{intro:prop_projective_homogeneous_space}, there exists a finite collection of finite étale Henselian local $R$-algebras $\{R_i\}_{i\in I}$ of collectively coprime degrees such that $(E/B)_{R_i}\to C_{R_i}$ admits a section for every $i\in I$. Equivalently, $E_{R_i}$ admits a reduction of structure group to $B\times_C C_{R_i}$.
The proof when $\kappa$ is large is the same, using that $E/B \to C$ admits a section.
\end{proof}

\section{Relative compactifications of torsors}\label{sec:6}

\begin{defin}\label{def:equivariant-compactification}
    Let $S$ be a scheme, let $G$ be a smooth finitely presented $S$-group, and let $X$ be a smooth finitely presented algebraic space over $S$ on which $G$ acts over $S$. By definition, a \emph{$G$-equivariant compactification} of $X$ is a smooth proper finitely presented algebraic space $\cl{X}$ over $S$, together with a $G$-action on $\cl{X}$ and a $G$-equivariant fiberwise dense open embedding $X\hookrightarrow \cl{X}$ over $S$.
\end{defin} 

Let $S$ be a scheme, and let $G$ be a smooth finitely presented $S$-group. In this section, under the assumption that the radical torus $\mathrm{rad}(G)$ of $G$ is isotrivial, we construct for every $G$-torsor $E\to S$ a $G$-equivariant compactification $\cl{E}\to S$. 

Consider the $(G \times_S G)$-action on $G$ given by left and right multiplication. By a theorem of Nath \cite[Theorem~1.2]{nath2026compactification}, if $G$ is semisimple, or if $G$ is an isotrivial reductive group, then $G$ admits a $(G \times_S G)$-equivariant compactification $\cl{G}\to S$ such that $\cl{G}$ is projective over $S$.

\begin{lemma}\label{iso_trivial_quasi_split_group}
Let $S$ be a scheme and $G$ a quasi-split reductive $S$-group. 
If the radical torus $\mathrm{rad}(G)$ is isotrivial, then $G$ is isotrivial.
\end{lemma}

\begin{proof}
Let $\tilde{G}\coloneqq G^{\mathrm{sc}} \times_S \mathrm{rad}(G)$, where $G^{\mathrm{sc}}$ is the simply connected cover of the derived subgroup of $G$. By \cite[XXII 6.2.3]{SGA3}, there exists a central short exact sequence
\[
1 \longrightarrow \mu(G) \longrightarrow \tilde{G} 
\longrightarrow G \longrightarrow 1,
\]
where $\mu(G)$ is a finite flat $S$-group of multiplicative type.
Observe that $G^{\mathrm{sc}}$ is quasi-split: indeed, letting $B\subset G$ be a Borel subgroup of $G$, by \cite[XIV Proposition 4.9]{SGA3} the inverse image $\tilde{B}\subset\tilde{G}$ is a Borel subgroup of $\tilde{G}$, and hence the image of $\tilde{B}$ in $G^{\mathrm{sc}}$ is a Borel subgroup of $G^{\mathrm{sc}}$. It now follows from \cite[XXIV Prop.~3.13]{SGA3} that $G^{\mathrm{sc}}$ contains a maximal torus of the form $\mathrm{Res}_{S'/S}(\mathbb{G}_{m,S'})$ for some finite \'etale cover $S' \to S$, which implies that $G^{\mathrm{sc}}$ is isotrivial. Since $\mathrm{rad}(G)$ is isotrivial by assumption, we conclude that $G^{\mathrm{sc}} \times_S \mathrm{rad}(G)$ is isotrivial. Finally, as $G$ is the quotient of $\tilde{G}$ by $\mu(G)$, it is also isotrivial.
\end{proof}

\begin{prop}\label{compactification_G}
    Let $S$ be a scheme, and let $G$ be a reductive $S$-group whose radical torus
    $\mathrm{rad}(G)$ is isotrivial.

   \begin{enumerate}
       \item There exists a smooth proper $S$-algebraic space $\overline{G}$, endowed with a left and a right action of $G$,
    together with a $G \times_S G$-equivariant open immersion $G \hookrightarrow \overline{G}$ whose image is fiberwise dense.
    \item If $S$ is a smooth projective family of integral curves over a scheme $B$, then we can choose $\overline{G}$ such that for every field $k$ and every map $\Spec k \to B$, $\overline{G}_{k} \to S_k$ is projective.
   \end{enumerate}
\end{prop}
\begin{proof}
(1)  By \cite[XXIV Cor. 3.12]{SGA3}, there exists a quasi-split reductive $S$-group
$G'$ such that $G$ is an inner form of $G'$. Let $E$ be a
$G'_{\mathrm{ad}}$-torsor over $S$ such that
$G = E \times^{G'_{\mathrm{ad}}} G'$, where
$G'_{\mathrm{ad}}$ acts on $G'$ by conjugation. As $\mathrm{rad}(G')$ is central, we have
$\mathrm{rad}(G') \cong \mathrm{rad}(G)$, and so $\mathrm{rad}(G')$
is isotrivial. By \Cref{iso_trivial_quasi_split_group}, $G'$ is isotrivial,
and hence by \cite[Theorem~1.2]{nath2026compactification}, there exists a
$G' \times_S G'$-equivariant compactification $\overline{G'}$
of $G'$.
We note that the conjugation action of $G'$ on $\overline{G'}$
factors through $G'_{\mathrm{ad}}$, since the center acts
trivially on the fiberwise dense open subscheme $G'$. Indeed, consider the
two morphisms
\[
    \sigma \colon Z(G') \times_S \overline{G'}
    \longrightarrow \overline{G'},
    \qquad
    p \colon Z(G') \times_S \overline{G'}
    \longrightarrow \overline{G'},
\]
given respectively by conjugation and by projection. These two morphisms coincide on
the fiberwise dense open subscheme $Z(G') \times_S G'$ of
$Z(G') \times_S \overline{G'}$. Since $\overline{G'}$
is projective, they therefore coincide on the schematic closure of
$Z(G') \times_S G'$ in
$Z(G') \times_S \overline{G'}$. As $Z(G')$ is flat and
quasi-compact over $S$, and since schematic closure commutes with flat quasi-compact base
change, it follows that $\sigma$ and $p$ coincide on all of
$Z(G') \times_S \overline{G'}$.
We can then define
$\overline{G} \coloneqq E \times^{G'_{\mathrm{ad}}} \overline{G'}$.
The open immersion $G \hookrightarrow \overline{G}$ is a twist of the
$G' \times G'$-equivariant open immersion
$G' \hookrightarrow \overline{G'}$ by the
$\mathrm{Aut}(\overline{G'})$-torsor
$E \times^{G'_{\mathrm{ad}}} \mathrm{Aut}(\overline{G'})$.
In particular, $G$ is fiberwise dense in $\overline{G}$, and the
$G' \times G'$-action on $\overline{G'}$ twists into a
$G \times_S G$-action on $\overline{G}$.

(2) Suppose moreover that $S$ is a family of integral curves over $B$, and let $\Spec k \to B$ be a morphism. Let $\overline{k}$ be an algebraic closure of $k$. By \cite[Chap.~III, \S2.2, Proposition~11]{Serre1997GaloisCohomology}, the field $k(S_{\overline{k}})$ has the $C_1$ property. Therefore, by \cite[Chap.~III, \S2.3, Theorem 1' and Remark (1) p.133]{Serre1997GaloisCohomology}, every $(G'_{\mathrm{ad}})_{\overline{k}}$-torsor over $S_{\overline{k}}$ is generically trivial, and therefore, by the Grothendieck--Serre conjecture (see \cite{fedorov2015proof}), it is Zariski-locally trivial. Hence $E_{\overline{k}}$ is Zariski-locally trivial.
It follows that $\overline{G}_{\overline{k}}$ is a smooth projective scheme over $S_{\overline{k}}$, and therefore also a smooth projective scheme over $\overline{k}$. Since projectivity descends under field extensions \cite[Corollaire~6.6.5]{ega2}, we deduce that $\overline{G}_{k}$ is a smooth projective scheme over $k$, and hence projective over $S_k$.
\end{proof}

\begin{cor}\label{compactification_torsor}
Let $S$ be a scheme, let $G$ be a reductive $S$-group whose radical torus
$\mathrm{rad}(G)$ is isotrivial, and let $E$ be a $G$-torsor over $S$.

\begin{enumerate}
\item There exists a smooth proper
$S$-algebraic space $\overline{E}$, endowed with a right action of $G$,
together with a $G$-equivariant open immersion
$E \hookrightarrow \overline{E}$ whose image is fiberwise dense.

\item If $S$ is a smooth projective family of integral curves over a scheme $B$,
then we can choose $\overline{E}$ such that for every field $k$ and every map
$\Spec k \to B$, the morphism
$\overline{E}_{k} \to S_k$ is projective.
\end{enumerate}
\end{cor}
\begin{proof}
    Let $\overline{G}$ be a proper $S$-algebraic space as in \Cref{compactification_G}.
    We define $\overline{E} \coloneqq E \times^G \overline{G}$, where $G$ acts on $\overline{G}$ on the
    left via the given left action extending left multiplication on $G$. Then $\overline{E}$ is a proper $S$-algebraic space, and the natural map $E \hookrightarrow \overline{E}$ is a fiberwise dense $G$-equivariant open immersion. The second part follows by the same argument as in \Cref{compactification_G}.
\end{proof}

\begin{rmk}
    If $S$ is quasi-compact and quasi-separated, the compactifications in \Cref{compactification_G,compactification_torsor} can in fact be chosen to be projective $S$-schemes; see \Cref{projectivity_compactification_G} and \Cref{compactification-is-projective}.
\end{rmk}

\section{Proof of Theorems \ref{thm:main}, \ref{thm:large-field} and \ref{thm:main-average}}\label{sec:7}

Theorems \ref{thm:main}, \ref{thm:large-field} and \ref{thm:main-average} follow from the next theorem.

\begin{thm}\label{thm:main-technical}
  Let $R$ be a Henselian local ring with residue field $\kappa$. Let $C \to \Spec R$ be a smooth projective morphism with geometrically connected fibers of pure dimension $1$. Let $G$ be a reductive group scheme over $C$  such that $\mathrm{rad}(G)$ is isotrivial, and let $E$ be a $G$-torsor over $C$ such that $E_{\kappa} \to C_{\kappa}$ is Zariski-locally trivial. There exists a finite collection of finite \'etale Henselian extensions $\{R_i/R\}_{i\in I}$ of collectively coprime degrees and a Zariski cover $(U_j)_{j\in J}$ of $C$ such that, for each $i\in I$ and $j \in J$, the base change $E_{R_i} \to C_{R_i}$ is trivial over $(U_j)_{R_i}$. In particular, the $G$-torsors $E_{R_i} \to C_{R_i}$ are Zariski-locally trivial.
  
  Suppose further that at least one of the following conditions is satisfied.    
    \begin{enumerate}
        \item[(a)] The field $\kappa$ is large.
        \item[(b)] The $C$-group $G$ is commutative (that is, a $C$-torus).
        \item[(c)] The kernel $\mu(G)$ of the central isogeny $G^{\mathrm{sc}}\times _C\mathrm{rad}(G)\to G$, where $G^{\mathrm{sc}}$ is the simply connected cover of the derived subgroup of $G$ and $\mathrm{rad}(G)$ is the radical $C$-torus of $G$, is \'etale over $C$.
    \end{enumerate}
Then $E \to C$ is Zariski-locally trivial. 
\end{thm}

\begin{proof}[Proof of \Cref{thm:main-technical}]
   Let $\overline{G}$ be a $G \times_C G$-equivariant compactification of $G$ over $C$. Let $\overline{E} \coloneqq E \times^G \overline{G}$ be the contracted product as in \Cref{compactification_torsor}. Then $\overline{E}$ is a smooth proper $C$-algebraic space containing $E$ as a fiberwise-dense open subscheme and $\overline{E}_{\kappa} \to C_{\kappa}$ is a smooth projective map of schemes. For every field extension $k/\kappa$ and every $k$-point $c \in C_{\kappa}(k)$, the $k$-variety $(\overline{E}_{\kappa})_c$ is separably rationally connected. Indeed, by \Cref{src_invariant_base_change}, in order to prove this we may assume that $k$ is algebraically closed. In this case, $G_c$ is split and the torsor $(E_{\kappa})_c$ is trivial and hence in particular rational, so that $(\overline{E}_{\kappa})_c = (\overline{G}_{\kappa})_c$ is rational and therefore separably rationally connected by \Cref{unirational_src}. Thus the map $\cl{E}_{\kappa} \to C_{\kappa}$ is a smooth projective map of schemes with separably rationally connected fibers.

    For the first assertion, let $z \in C_{\kappa}$ be a closed point and let $s_z \colon C_\kappa \dashrightarrow E_{\kappa}$ be a rational section of $E_\kappa \to C_\kappa$ defined on an open subset of $C_\kappa$ containing $z$. This section extends to a section $\overline{s}_z \colon C_\kappa \to \overline{E}_\kappa$. We apply \Cref{cor_prop_principal_Weil}(1) to $X = \overline{E}$ and $Z = \{z\}$, which yields finite \'etale extensions $\{R_{z,i}\}_{i \in I_z}$ of $R$ of collectively coprime degrees and sections $s_{z,i} \colon C_{R_{z,i}} \to \overline{E}_{R_{z,i}}$ satisfying $s_{z,i}|_{\{z\}_{\kappa_{z,i}}} = s_z|_{\{z\}_{\kappa_{z,i}}}$, where $\kappa_{z,i}$ denotes the residue field of $R_{z,i}$.
For each $i \in I_z$, the section $s_{z,i}$ trivializes the torsor
$E_{R_{z,i}} \to C_{R_{z,i}}$ over an open subset
$U_{z,i} \subset C_{R_{z,i}}$ containing the closed subset
$\{z\}_{\kappa_{z,i}}$. Set
\[
U_z \coloneqq \bigcap_{i \in I_z}
\bigl(C \setminus \pi_{z,i}(C_{R_{z,i}} \setminus U_{z,i})\bigr),
\]
where $\pi_{z,i} \colon C_{R_{z,i}} \to C$ is the natural finite \'etale
projection. Then $U_z \subset C$ is open, contains $z$, and each section
trivializes $E_{R_{z,i}}$ over $(U_z)_{R_{z,i}}$. Since $C_\kappa$ is
quasi-compact, finitely many $U_{z_1}, \dots, U_{z_m}$ cover $C_\kappa$, hence
cover $C$. For every tuple $(i^{(z_j)})_j \in \prod_j I_{z_j}$ there is a
decomposition
\[
\bigotimes_{1 \leqslant j \leqslant m} R_{z_j, i^{(z_j)}}
\cong \prod_{i} R_{i^{(z_1)}, \dots, i^{(z_m)}, i},
\]
into finite \'etale Henselian extensions of $R$. Letting $\{R_i\}$ be these
factors and $(U_j)_{1 \leqslant j \leqslant m} = (U_{z_j})_{1 \leqslant j \leqslant m}$ yields the first assertion. We now turn to the proof of the second assertion.

Assume (a). In this case, the proof of the second assertion is identical to the proof of the first assertion, using \Cref{cor_prop_principal_Weil}(2) in place of \Cref{cor_prop_principal_Weil}(1). 

Assume (b). By the first assertion, there exists a finite collection of finite \'etale extensions $\{R_i/R\}_{i\in I}$ of collectively coprime degrees and a Zariski open cover $(U_j)_{1 \leqslant j \leqslant J}$ of $C$ such that for every $i \in I$ the torsor $E_{R_i} \to C_{R_i}$ is trivial over $(U_j)_{R_i}$.
Let $\eta_j \in H^1(U_j, G)$ be the class of $E|_{U_j}$. By a restriction–corestriction argument for the abelian group $H^1(U_j,G)$, one obtains $[R_i : R] \cdot \eta_j = 0$ for all $i, j$. Since the integers $([R_i : R])_{i\in I}$ are collectively coprime, it follows that $\eta_j = 0$ for every $j$, which concludes the proof.

Finally, assume (c). First consider the case where the $G$-torsor $E \to C$ is trivial on the closed fiber of $R$. We follow the argument of \cite[Theorem~7.1]{Gille2021}. Consider the commutative diagram of pointed sets
\[
\begin{tikzcd}
H^1(C,\mu(G)) \arrow[r] \arrow[d]
& H^1(C,G^{\mathrm{sc}}) \times H^1(C,\mathrm{rad}(G)) \arrow[r] \arrow[d]
& H^1(C,G) \arrow[r] \arrow[d]
& H^2(C,\mu(G)) \arrow[d] \\
H^1(C_\kappa,\mu(G)_{\kappa}) \arrow[r]
& H^1(C_\kappa,(G^{\mathrm{sc}})_{\kappa}) \times H^1(C_\kappa,\mathrm{rad}(G)_{\kappa}) \arrow[r]
& H^1(C_\kappa,G_{\kappa}) \arrow[r]
& H^2(C_\kappa,\mu(G)_{\kappa})
\end{tikzcd}
\]
in which the rows are exact. By the proper base change theorem in \'etale cohomology \cite[XII.5.5.(iii)]{SGA4compact}, the maps $H^i(C,\mu(G)) \to H^i(C_\kappa,\mu(G))$ are isomorphisms for all $i \geqslant 0$. We set
\[
\begin{aligned}
K_{\mathrm{rad}} &\coloneqq \ker\!\big(H^1(C,\mathrm{rad}(G)) \longrightarrow H^1(C_\kappa,\mathrm{rad}(G)_\kappa)\big), \\
K_{\mathrm{sc}} &\coloneqq \ker\!\big(H^1(C,G^{\mathrm{sc}}) \longrightarrow H^1(C_\kappa,(G^{\mathrm{sc}})_{\kappa})\big), \\
K_G &\coloneqq \ker\!\big(H^1(C,G) \longrightarrow H^1(C_\kappa,G_{\kappa})\big).
\end{aligned}
\]
It follows that the natural map $K_{\mathrm{sc}} \times K_{\mathrm{rad}} \to K_G$ is surjective. We may apply case~(b) to the commutative group $\mathrm{rad}(G)$, and hence $K_{\mathrm{rad}}$ consists of Zariski-locally trivial torsors. On the other hand, $K_{\mathrm{sc}}$ consists of Zariski-locally trivial torsors by \cite[Proposition~2.2(3) and Theorem~2.4]{Gille2021}. Hence the same holds for $K_G$. Since the class of $E$ lies in $K_G$, it follows that $E$ is Zariski-locally trivial, as desired.
To prove the result in general, namely when the $G_\kappa$-torsor $E_\kappa$ is Zariski-locally trivial, the proof of \cite[Theorem~7.2]{Gille2021} applies.
\end{proof}

\begin{proof}[Proof of \Cref{cor:local-global}]
    Since $R$ is a discrete valuation ring, it is Noetherian and regular, hence geometrically unibranch. It follows that the torus $\mathrm{rad}(G)$ is isotrivial; see \cite[Expos\'e X, Th\'eor\`eme 5.16]{SGA3IInew}.  It now suffices to replace \cite[Theorem 7.2]{Gille2021} by \Cref{thm:main} in the proof of \cite[Corollary 7.8]{Gille2021}. 
\end{proof}

\appendix
\section{Projective compactification of torsors}\label{sec:appendix}

Let $S$ be a scheme, let $G$ be a reductive $S$-group whose radical torus $\mathrm{rad}(G)$ is isotrivial, and recall the definition of an equivariant compactification from \Cref{def:equivariant-compactification}. We now refine \Cref{compactification_G} by proving that $G$ admits a projective (in particular, schematic) $(G\times_SG)$-equivariant compactification. 

\begin{thm}\label{projectivity_compactification_G}
    Let $S$ be a quasi-compact quasi-separated scheme, and let $G$ be a reductive $S$-group such that the radical torus $\mathrm{rad}(G)$ is isotrivial.
    Then there exists a $(G \times_S G)$-equivariant compactification $\overline{G}$ endowed with a relatively very ample line bundle $\mathcal{L}$ equipped with a
    $(G \times_S G)$-linearization. In particular, $\overline{G}$ is projective over $S$.
\end{thm}

This is a consequence of the results of \cite{nath2026compactification}. We begin by recalling the construction of the compactification in \cite{nath2026compactification} and then use it to prove \Cref{projectivity_compactification_G}. Let $S$ be a scheme, let $G$ be an isotrivial reductive $S$-group, and let $S' \to S$ be a finite étale Galois cover splitting $G$, with Galois group $\Gamma$. Then $G$ is an inner form of a quasi-split reductive group $G'$ endowed with a quasi-pinning over $S$. The uniqueness of this form implies that $G'$ is itself isotrivial.
Since finite \'etale descent is effective for projective schemes, it suffices to prove the existence of the compactification for the quasi-split isotrivial group $G'$.

Let $\mathbf{G}$ denote the split form of $G'$. The idea in
\cite{nath2026compactification} is first to construct a compactification of
$\mathbf{G}_{S'}$ by means of its Vinberg monoid (see \cite[Section 2.1]{nath2026compactification}) together with a fan
$\Sigma \subset \mathbf{X}_{\bullet}(\mathbf{T})$ in the cocharacter lattice of a maximal
torus $\mathbf{T}$, assumed to be $\Gamma$-equivariant, smooth, and projective (see \cite[Lemma 3.2.1]{nath2026compactification}). This
allows one to construct an affine monoid $V_{\mathbf{G},\beta}$ whose group of units is
$(\mathbb{G}_m)^I \times \mathbf{G}_{S'}$, where $I$ denotes the set of rays of the fan
$\Sigma$. Nath then defines a $(\mathbb{G}_m)^I$-equivariant line bundle $\mathcal{M}$ on
$V_{\mathbf{G},\beta}$. The underlying line bundle is trivial, and the linearization is
determined by a character $\rho \colon (\mathbb{G}_m)^I \to \mathbb{G}_m$. That is, the character
$\rho \colon (\mathbb{G}_m)^I \to \mathbb{G}_m$ gives a map
\[
    f \colon (\mathbb{G}_m)^I \to \operatorname{Aut}(\mathbb{A}^1_{V_{\mathbf{G},\beta}}),
\]
and the $(\mathbb{G}_m)^I$-linearization is then taken to be
\begin{equation}\label{linearization_character}
    (\mathbb{G}_m)^I \times \mathbb{A}^1_{V_{\mathbf{G},\beta}}
    \to \mathbb{A}^1_{V_{\mathbf{G},\beta}},
    \quad (u, l) \mapsto f(u) \circ (\operatorname{id}, u\cdot)(l)
\end{equation}
where $(\operatorname{id}, u\cdot)\colon \mathbb{A}^1_{V_{\mathbf{G},\beta}}= \mathbb{A}^1_{S'} \times_{S'} V_{\mathbf{G},\beta} \to \mathbb{A}^1_{V_{\mathbf{G},\beta}}$ is given by $(a,v) \mapsto (a,u\cdot v)$.
We will also denote this linearization by $\rho$. The compactification is then defined as the GIT quotient
\[
    \overline{\mathbf{G}}_{S'} \coloneqq V_{\mathbf{G},\beta} /\!/_{\rho}\, (\mathbb{G}_m)^I
\]
(see \cite[Section~3.3]{nath2026compactification}). It inherits an action of $\Gamma$
such that the open immersion $\mathbf{G}_{S'} \hookrightarrow \overline{\mathbf{G}}_{S'}$
is $\Gamma$-equivariant.

\begin{lemma}\label{G_X_G_linearization}
    Let $S$ be a scheme and let $\mathbf{G}$ be a split reductive $S$-group. Let $\mathcal{L}$ be the trivial line
    bundle on $V_{\mathbf{G},\beta}$ equipped with the $(\mathbb{G}_m)^I$-linearization defined by a character $\rho\colon (\mathbb{G}_m)^I \to \mathbb{G}_m$. Then $\mathcal{L}$ can be endowed
    with a $(\mathbf{G} \times_S \mathbf{G})$-linearization commuting with the $(\mathbb{G}_m)^I$-linearization induced by $\rho$ described in \eqref{linearization_character}, such that the $\mathbf{G}$-linearization induced on $\mathcal{L}$ by the map $\mathbf{G} \to \mathbf{G}\times_S \mathbf{G}$, $g \mapsto (g,g)$ factors through a $\mathbf{G}_{\mathrm{ad}}$-linearization.
\end{lemma}

\begin{proof}
    Since $\mathcal{L}$ is the trivial line bundle, we have $\mathbb{V}(\mathcal{L}) =
    \mathbb{A}^1_{V_{\mathbf{G},\beta}}$. We may therefore equip $\mathcal{L}$ with the trivial linearization
    with respect to $\mathbf{G} \times_S \mathbf{G}$, denoted $\rho_{\mathrm{triv},\mathbf{G}}$.

    Since the action of $(\mathbb{G}_m)^I$ on $\mathbb{A}^1_{V_{\mathbf{G},\beta}}$
    commutes with the action of $\mathbf{G}\times_S \mathbf{G}$, for every $S$-scheme $U$, every
    $(g_1,g_2) \in (\mathbf{G}\times_S\mathbf{G})(U)$, every $u \in (\mathbb{G}_m)^I(U)$, and every $l \in \mathbb{A}^1_{V_{\mathbf{G},\beta}}(U)$, we have
    \[
        \rho_{\mathrm{triv},\mathbf G}(g_1,g_2)\bigl((\operatorname{id}, u\cdot)(l)\bigr)
        =(\operatorname{id}, u\cdot)
        \bigl(\rho_{\mathrm{triv},\mathbf{G}}((g_1, g_2))(l)\bigr).
    \]
    Moreover, for every $(g_1, g_2) \in (\mathbf{G}\times_S\mathbf{G})(U)$,
    every $u \in (\mathbb{G}_m)^I(U)$, and every $l \in \mathbb{A}^1_{V_{\mathbf{G},\beta}}(U)$, we have
    \[
        \rho_{\mathrm{triv},\mathbf{G}}((g_1, g_2))\bigl(f(u)(l)\bigr)
        =
        f(u)\bigl(\rho_{\mathrm{triv},\mathbf{G}}((g_1, g_2))(l)\bigr).
    \]
    This yields a $\mathbf{G} \times_S \mathbf{G}$-linearization of
    $\mathcal{L}$ commuting with the $(\mathbb{G}_m)^I$-linearization induced by $\rho$ described in \eqref{linearization_character}.
    Moreover, the $\mathbf{G}$-linearization induced on $\mathcal{L}$ by the map
$\mathbf{G} \to \mathbf{G} \times_S \mathbf{G}$, $g \mapsto (g, g)$,
factors through a $\mathbf{G}_{\mathrm{ad}}$-linearization, as the action of
$\mathbf{G}$ on $V_{\mathbf{G},\beta}$ induced by the same map factors through
a $\mathbf{G}_{\mathrm{ad}}$-action.
\end{proof}

\begin{lemma}
\label{line_bundle_quasi_split}
    Let $S$ be a quasi-compact quasi-separated scheme. Let $G'$ be an isotrivial quasi-split reductive $S$-group
    equipped with a quasi-pinning (\emph{quasi-épinglage} in French). There exists a $(G'\times_S G')$-equivariant compactification
    $\overline{G'}$ endowed with a relatively very ample line bundle $\mathcal{L}'$ equipped with a
    $(G' \times_S G')$-linearization, such that the $G'$-linearization induced by the map
    \[
        G' \to G' \times_S G', \quad g \mapsto (g,g)
    \]
    factors through a $G'_{\mathrm{ad}}$-linearization.
\end{lemma}

\begin{proof}
By absolute Noetherian approximation \cite[Proposition 01ZA]{stacks-project}, we may assume that $S$ is Noetherian.

Let $S' \to S$ be a finite \'etale Galois cover with Galois group $\Gamma$ splitting $G'$. Let $\mathbf{G}$ be the split reductive $S$-group corresponding to $G'$. By descent, this induces an action of $\Gamma$ on $\mathbf{G}_{S'}$, which extends to $\overline{\mathbf{G}}_{S'}$ by \cite[Section~3.3]{nath2026compactification}, making the $(\mathbf{G}_{S'} \times_{S'} \mathbf{G}_{S'})$-equivariant open immersion
\[
\mathbf{G}_{S'} \hookrightarrow \overline{\mathbf{G}}_{S'}
\]
compatible with the action of $\Gamma$. Let $p \colon V_{\mathbf{G}_{S'},\beta} \to S'$ be the structural morphism. We have
\[
    \overline{\mathbf{G}}_{S'}
    =
    \operatorname{Proj}\!\left(
        \bigoplus_{n \ge 0}
        p_*\bigl( \mathcal{M}^{\otimes n}\bigr)^{(\mathbb{G}_m)^I}
    \right).
\]
By \Cref{G_X_G_linearization}, there exists a $(\mathbf{G}_{S'} \times_{S'} \mathbf{G}_{S'})$-linearization of $\mathcal{M}$ commuting with the $(\mathbb{G}_m)^I$-linearization. Since this linearization commutes with the $(\mathbb{G}_m)^I$-linearization, it induces an action of $\mathbf{G}_{S'} \times_{S'} \mathbf{G}_{S'}$ on
\[
    \mathcal{A} \coloneqq \bigoplus_{n \ge 0}
    p_*\bigl( \mathcal{M}^{\otimes n}\bigr)^{(\mathbb{G}_m)^I},
\]
and therefore, for every $n \in \mathbb{Z}$, the tautological sheaf $\mathcal{O}(n)$ is endowed with a $(\mathbf{G}_{S'} \times_{S'} \mathbf{G}_{S'})$-linearization. Moreover, by the last point of \Cref{G_X_G_linearization}, these linearizations can be chosen so that the $\mathbf{G}_{S'}$-linearization induced by the morphism
\[
    \mathbf{G}_{S'} \to \mathbf{G}_{S'} \times_{S'} \mathbf{G}_{S'}, \quad g \mapsto (g,g)
\]
factors through a $(\mathbf{G}_{S'})_{\mathrm{ad}}$-linearization.
It follows that, for every $n \in \mathbb{Z}$ and every $\gamma \in \Gamma$, $\gamma^*\mathcal{O}(n)$ on $\overline{\mathbf{G}}_{S'}$ is also endowed with a $(\mathbf{G}_{S'} \times_{S'} \mathbf{G}_{S'})$-linearization.

By \cite[Theorem 2(i)]{Seshadri1977GeometricReductivity}, the graded $\mathcal{O}_{S'}$-algebra
$\mathcal{A}$
is of finite type. Hence, for $n$ sufficiently divisible, $\mathcal{O}(n)$ is a relatively very ample line bundle.
Setting
\[
    \mathcal{L}'_{\Gamma} \coloneqq \bigotimes_{\gamma \in \Gamma} \gamma^*\mathcal{O}(n),
\]
we obtain a line bundle on $\overline{\mathbf{G}}_{S'}$ equipped with a $\Gamma$-descent datum and a $\Gamma$-equivariant $(\mathbf{G}_{S'} \times_{S'} \mathbf{G}_{S'})$-linearization, with the additional property that the induced $\mathbf{G}_{S'}$-linearization via
$g \mapsto (g,g)$ factors through a $(\mathbf{G}_{S'})_{\mathrm{ad}}$-linearization. By \cite[Section~3.3]{nath2026compactification}, the descent datum on $\overline{\mathbf{G}}_{S'}$ is effective and defines a $(G' \times_S G')$-compactification $\overline{G'}$, which is a projective $S$-scheme. The $(\mathbf{G}_{S'} \times_{S'} \mathbf{G}_{S'})$-equivariant open immersion
$\mathbf{G}_{S'} \hookrightarrow \overline{\mathbf{G}}_{S'}$ moreover descends to a $G' \times_S G'$-equivariant open immersion
$G' \hookrightarrow \overline{G'}$. Finally, since $\overline{G'}$ is obtained by descent along the $\Gamma$-action, the line bundle $\mathcal{L}'_{\Gamma}$ descends to a line bundle $\mathcal{L}'$ on $\overline{G'}$ equipped with a $(G' \times_S G')$-linearization such that the $G'$-linearization induced by the map
$G' \to G' \times_S G', \quad g \mapsto (g,g)$
factors through a $G'_{\mathrm{ad}}$-linearization. Since relative very ampleness is fppf-local on the base, $\mathcal{L}'$ is relatively very ample.
\end{proof}

\begin{proof}[Proof of \Cref{projectivity_compactification_G}]
Let $G'$ be a quasi-split group equipped with a quasi-pinning such that $G$ is an
inner form of $G'$. Let $P$ be the $G'_{\mathrm{ad}}$-torsor over $S$ such that
$G \cong P \times^{G'_{\mathrm{ad}}} G'$, where $G'_{\mathrm{ad}}$ acts on $G'$ by conjugation. Then, by the proof of \Cref{compactification_G},
the group $G'$ is isotrivial.
Define $\overline{G} \coloneqq P \times^{G'_{\mathrm{ad}}} \overline{G'}$, where
$\overline{G'}$ is the compactification constructed in \Cref{line_bundle_quasi_split}.
This defines a proper $S$-algebraic space.
Since the $G'$-linearization of $\mathcal{L}'$ factors through $G'_{\mathrm{ad}}$, we may twist the line bundle $\mathcal{L}'$ of
\Cref{line_bundle_quasi_split} by $P$. This defines a relatively very ample line bundle on $\overline{G}$ equipped with a $G \times_S G$-linearization. In particular, by \cite[Lemma 03D4]{stacks-project}, $\overline{G}$ is a scheme.
Since $\overline{G}$ is proper over $S$ and carries a relatively very ample line bundle, it follows that $\overline{G}$ is a projective $S$-scheme.
\end{proof}

As a consequence of \Cref{projectivity_compactification_G}, we may refine \Cref{compactification_torsor} as follows.

\begin{cor}\label{compactification-is-projective}
    Let $S$ be a quasi-compact quasi-separated scheme, and let $G$ be a reductive $S$-group scheme with isotrivial
    radical torus. Let $E$ be a $G$-torsor over $S$. Then there exists a $G$-equivariant compactification $\overline{E}$ of $E$ endowed with a relatively very ample line bundle $\mathcal{L}_{E}$ equipped with a
    $G$-linearization for the right action of $G$ on $\overline{E}$. In particular,
    $\overline{E}$ is a projective $S$-scheme.
\end{cor}

\begin{proof}
Let $\overline{G}$ be a $(G \times_S G)$-equivariant compactification of
$G$ given by \Cref{projectivity_compactification_G}. In particular,
$\overline{G}$ is a projective $S$-scheme, and it is equipped with a
line bundle $\mathcal{L}$ endowed with a $(G \times_S G)$-linearization.
Let $\overline{E} \coloneqq E \times^{G} \overline{G}$ be the contracted
product, where $G$ acts on $\overline{G}$ through the first factor
$G \times_S \{e_G\}$ of $G \times_S G$. By the same argument as in the
proof of \Cref{projectivity_compactification_G}, descent yields a
$S$-relatively very ample line bundle on $\overline{E}$, equipped with a
$G$-linearization for the residual action of the second factor
$\{e_G\} \times_S G$.
\end{proof}

\section*{Acknowledgments}

We thank Paolo Bravi, Jean-Louis Colliot-Th\'el\`ene, Philippe Gille, J\'anos Koll\'ar, Alena Pirutka, Jason Starr, and Zhiyu Tian for helpful comments.

\newcommand{\etalchar}[1]{$^{#1}$}


\begin{thebibliography}{CTOP02}


\bibitem[AK03]{AraujoKollar2003}
Carolina Araujo and J\'anos Koll\'ar.
\newblock Rational curves on varieties.
\newblock In {\em Higher dimensional varieties and rational points ({B}udapest, 2001)}, volume~12 of {\em Bolyai Soc. Math. Stud.}, pages 13--68. Springer, Berlin, 2003.

\bibitem[Aok06a]{aoki2006hom}
Masao Aoki.
\newblock Hom stacks.
\newblock {\em Manuscripta Math.}, 119(1):37--56, 2006.

\bibitem[Aok06b]{aoki2006hom-erratum}
Masao Aoki.
\newblock Erratum: ``{H}om stacks'' [{M}anuscripta {M}ath. {\bf 119} (2006), no. 1, 37--56.
\newblock {\em Manuscripta Math.}, 121(1):135, 2006.

\bibitem[BL94]{beauville1994conformal}
Arnaud Beauville and Yves Laszlo.
\newblock Conformal blocks and generalized theta functions.
\newblock {\em Comm. Math. Phys.}, 164(2):385--419, 1994.

\bibitem[CTOP02]{ColliotTheleneOjangurenParimala2002}
Jean-Louis Colliot-Th{\'e}l{\`e}ne, Manuel Ojanguren, and Raman Parimala.
\newblock Quadratic forms over fraction fields of two-dimensional {H}enselian rings and {B}rauer groups of related schemes.
\newblock In {\em Algebra, Arithmetic and Geometry, {P}art {I}, {II} ({M}umbai, 2000)}, volume~16 of {\em Tata Inst. Fund. Res. Stud. Math.}, pages 185--217. Tata Inst. Fund. Res., Bombay, 2002.

\bibitem[CTPS12]{colliot2012patching}
Jean-Louis Colliot-Th\'el\`ene, Raman Parimala, and Venapally Suresh.
\newblock Patching and local-global principles for homogeneous spaces over function fields of {$p$}-adic curves.
\newblock {\em Comment. Math. Helv.}, 87(4):1011--1033, 2012.

\bibitem[CTS21]{colliot2021brauer}
Jean-Louis Colliot-Th\'el\`ene and Alexei~N. Skorobogatov.
\newblock {\em The {B}rauer-{G}rothendieck group}, volume~71 of {\em Ergebnisse der Mathematik und ihrer Grenzgebiete. 3. Folge.}.
\newblock Springer, Cham, 2021.

\bibitem[Deb11]{Debarre2010GAeL}
Olivier Debarre.
\newblock Rational curves on algebraic varieties.
\newblock \url{https://www.math.ens.psl.eu/~debarre/NotesGAEL.pdf}, 2011.
\newblock Lecture notes for the GAeL XVIII conference, Coimbra, Portugal, June 6--11, 2010, and for the Semaine sp\'eciale Master de math\'ematiques, Universit\'e de Strasbourg, May 2--6, 2011.

\bibitem[DG70]{SGA3}
M.~Demazure and A.~Grothendieck, editors.
\newblock {\em Schémas en groupes}, volumes 151--153 of {\em Lecture Notes in Mathematics}.
\newblock Springer-Verlag, Berlin, 1970.
\newblock Séminaire de géométrie algébrique du Bois Marie 1962--64 (SGA 3).

\bibitem[DS95]{drinfeld1995b-structures}
V.~G. Drinfeld and Carlos Simpson.
\newblock {$B$}-structures on {$G$}-bundles and local triviality.
\newblock {\em Math. Res. Lett.}, 2(6):823--829, 1995.

\bibitem[FGI{\etalchar{+}}05]{fantechi2005fundamental}
Barbara Fantechi, Lothar G\"ottsche, Luc Illusie, Steven~L. Kleiman, Nitin Nitsure, and Angelo Vistoli.
\newblock {\em Fundamental algebraic geometry}, volume 123 of {\em Mathematical Surveys and Monographs}.
\newblock American Mathematical Society, Providence, RI, 2005.
\newblock Grothendieck's FGA explained.

\bibitem[FP15]{fedorov2015proof}
Roman Fedorov and Ivan Panin.
\newblock A proof of the {G}rothendieck-{S}erre conjecture on principal bundles over regular local rings containing infinite fields.
\newblock {\em Publ. Math. Inst. Hautes \'Etudes Sci.}, 122:169--193, 2015.

\bibitem[GD61]{ega2}
Alexander Grothendieck and Jean Dieudonn{\'e}.
\newblock {\em {\'E}l{\'e}ments de g{\'e}om{\'e}trie alg{\'e}brique {II}}, volume~8 of {\em Publications {M}ath{\'e}matiques}.
\newblock Institut des {H}autes {\'E}tudes {S}cientifiques, 1961.

\bibitem[GHS03]{graber2003families}
Tom Graber, Joe Harris, and Jason Starr.
\newblock Families of rationally connected varieties.
\newblock {\em J. Amer. Math. Soc.}, 16(1):57--67, 2003.

\bibitem[GLL13]{gabber2013index}
Ofer Gabber, Qing Liu, and Dino Lorenzini.
\newblock The index of an algebraic variety.
\newblock {\em Invent. Math.}, 192(3):567--626, 2013.

\bibitem[GPS21]{Gille2021}
P.~Gille, R.~Parimala, and V.~Suresh.
\newblock Local triviality for $G$-torsors.
\newblock {\em Mathematische Annalen}, 380(1):539--567, 2021.

\bibitem[Gro57]{grothendieck1957classification}
A.~Grothendieck.
\newblock Sur la classification des fibr\'es holomorphes sur la sph\`ere de {R}iemann.
\newblock {\em Amer. J. Math.}, 79:121--138, 1957.

\bibitem[GW23]{GortzWedhornalgebraic}
U.~G\"ortz and T.~Wedhorn.
\newblock {\em Algebraic geometry {II}: {C}ohomology of schemes---with examples and exercises}.
\newblock Springer Studium Mathematik---Master. Springer Spektrum, Wiesbaden, [2023] \copyright 2023.

\bibitem[Har77]{Hartshorne1977algebraic}
Robin Hartshorne.
\newblock {\em Algebraic geometry}, volume No. 52 of {\em Graduate Texts in Mathematics}.
\newblock Springer-Verlag, New York-Heidelberg, 1977.

\bibitem[Har10]{hartshorne2010deformation}
Robin Hartshorne.
\newblock {\em Deformation theory}, volume 257 of {\em Graduate Texts in Mathematics}.
\newblock Springer, New York, 2010.

\bibitem[Hei10]{heinloth2010uniformization}
Jochen Heinloth.
\newblock Uniformization of $\mathcal{G}$-bundles.
\newblock {\em Math. Ann.}, 347(3):499--528, 2010.

\bibitem[HHK09]{harbater2009applications}
David Harbater, Julia Hartmann, and Daniel Krashen.
\newblock Applications of patching to quadratic forms and central simple algebras.
\newblock {\em Invent. Math.}, 178(2):231--263, 2009.

\bibitem[HM82]{hazewinkel1982short}
Michiel Hazewinkel and Clyde~F. Martin.
\newblock A short elementary proof of {G}rothendieck's theorem on algebraic vectorbundles over the projective line.
\newblock {\em J. Pure Appl. Algebra}, 25(2):207--211, 1982.

\bibitem[Kol96]{kollar1996rational}
J.~Koll\'ar.
\newblock {\em Rational curves on algebraic varieties}, volume~32 of {\em Ergebnisse der Mathematik und ihrer Grenzgebiete. 3. Folge. A Series of Modern Surveys in Mathematics [Results in Mathematics and Related Areas. 3rd Series. A Series of Modern Surveys in Mathematics]}.
\newblock Springer-Verlag, Berlin, 1996.

\bibitem[Kol04]{Kollar2004Specialisation}
J\'anos Koll\'ar.
\newblock Specialization of zero cycles.
\newblock {\em Publ. Res. Inst. Math. Sci.}, 40(3):689--708, 2004.

\bibitem[Nat26]{nath2026compactification}
Ayan Nath.
\newblock Compactification of reductive group schemes.
\newblock {\em arXiv:2601.22462}, 2026.

\bibitem[Ols06a]{olsson2006deformation}
Martin~C. Olsson.
\newblock Deformation theory of representable morphisms of algebraic stacks.
\newblock {\em Math. Z.}, 253(1):25--62, 2006.

\bibitem[Ols06b]{Olsson2006underline}
Martin~C. Olsson.
\newblock {$\underline {\rm Hom}$}-stacks and restriction of scalars.
\newblock {\em Duke Math. J.}, 134(1):139--164, 2006.

\bibitem[Pop14]{pop2014little}
Florian Pop.
\newblock Little survey on large fields---old \& new.
\newblock In {\em Valuation theory in interaction}, EMS Ser. Congr. Rep., pages 432--463. Eur. Math. Soc., Z\"urich, 2014.

\bibitem[Ryd10]{Rydh2010Submersions}
David Rydh.
\newblock Submersions and effective descent of \'etale morphisms.
\newblock {\em Bull. Soc. Math. France}, 138(2):181--230, 2010.

\bibitem[Ser97]{Serre1997GaloisCohomology}
Jean-Pierre Serre.
\newblock {\em Galois Cohomology}.
\newblock Springer Monographs in Mathematics. Springer Berlin Heidelberg, Berlin, Heidelberg, first edition, 1997.
\newblock Originally published as a monograph.

\bibitem[SGA3$_{\mathrm{II}}$]{SGA3IInew}
Philippe Gille and Patrick Polo, editors.
\newblock {\em Sch\'emas en groupes ({SGA} 3). Tome II. Groupes de
  type multiplicatif, et structure des sch\'emas en groupes g\'en\'eraux}.
\newblock S\'eminaire de G\'eom\'etrie Alg\'ebrique du Bois Marie
  1962--64. A seminar directed by Michel Demazure and Alexander
  Grothendieck with the collaboration of Michael Artin, Jean-\'Etienne
  Bertin, Pierre Gabriel, Michel Raynaud, and Jean-Pierre Serre.
\newblock Revised and annotated edition of the 1970 French original, 2009.
\newblock \url{https://webusers.imj-prg.fr/~patrick.polo/SGA3/}.

\bibitem[SGA4$_{\mathrm{III}}$]{SGA4compact}
A.~Grothendieck, M.~Artin, and J.-L. Verdier, editors.
\newblock {\em Séminaire de Géométrie Algébrique du Bois Marie 1963--64: Théorie des topos et cohomologie étale des schémas}, volume 305 of {\em Lecture Notes in Mathematics}.
\newblock Springer, 1973.

\bibitem[Ses77]{Seshadri1977GeometricReductivity}
C.~S. Seshadri.
\newblock Geometric reductivity over arbitrary bases.
\newblock {\em Advances in Mathematics}, 26(3):225--274, 1977.

\bibitem[{Sta}18]{stacks-project}
The {Stacks Project Authors}.
\newblock \textit{Stacks Project}.
\newblock \url{https://stacks.math.columbia.edu}, 2018.

\bibitem[Sta23]{starr2023local}
Jason Starr.
\newblock Local triviality of torsors for relative reductive groups.
\newblock MathOverflow, \url{https://mathoverflow.net/q/460963}, 2023.
\newblock Version: 2023-12-24.

\bibitem[Zhu17]{zhu2017introduction}
Xinwen Zhu.
\newblock An introduction to affine {G}rassmannians and the geometric {S}atake equivalence.
\newblock In {\em Geometry of moduli spaces and representation theory}, volume~24 of {\em IAS/Park City Math. Ser.}, pages 59--154. Amer. Math. Soc., Providence, RI, 2017.

\end{thebibliography}
\end{document}